\documentclass[a4paper,oneside,12pt]{article}
\usepackage{amsfonts}
\usepackage{amsmath}
\usepackage{amssymb}
\usepackage{amsthm}
\usepackage{MnSymbol}
\usepackage{enumerate}
\usepackage{bm}
\usepackage[margin = 3cm]{geometry}
\usepackage{hyperref,url}
\usepackage{enumitem}
\usepackage[notref,notcite]{showkeys}
\usepackage{pstricks}
\usepackage{accents}
\usepackage{graphicx}
\usepackage{overpic}
\usepackage{subfigure}
\usepackage{soul}

\newtheorem{theorem}{Theorem}

\newtheorem{lemma}[theorem]{Lemma}

\newtheorem{remark}[theorem]{Remark}

\input{tcilatex}

\def\duav #1{\langle #1\rangle}
\def\grad {{\text{\bf grad}}}
\def\ub {{\bf u}}
\def\vb {{\bf v}}
\def\e {\varepsilon}
\def\dive {\text{div}}
\def\dn {\partial_{\bf n}}
\def\teta {\theta}
\begin{document}

\title{Collisions in shape memory alloys}
\author{Michel Fr\'emond \footnotemark[1] \and Michele Marino
\footnotemark[2] \and Elisabetta Rocca \footnotemark[3]}
\date{\today}
\maketitle

\begin{abstract}
We present here a model for instantaneous collisions in a solid made of
shape memory alloys (SMA) by means of a predictive theory which is based on
the introduction not only of macroscopic velocities and temperature, but
also of microscopic velocities responsible of the austenite-martensites
phase changes. Assuming time discontinuities for velocities, volume
fractions and temperature, and applying the principles of thermodynamics for
non-smooth evolutions together with constitutive laws typical of SMA, we end
up with a system of nonlinearly coupled elliptic equations for which we
prove an existence and uniqueness result in the 2 and 3 D cases. Finally, we
also present numerical results for a SMA 2D solid subject to an external
percussion by an hammer stroke.
\end{abstract}



\renewcommand{\thefootnote}{\fnsymbol{footnote}} \footnotetext[1]{%
Dipartimento di Ingegneria Civile e Informatica - Universit\`a di Roma Tor
Vergata, via del Politecnico 1, Roma, Italy (\texttt{michel.fremond@uniroma2.it}).}
\footnotetext[2]{%
Institut f{\"u}r Kontinuumsmechanik, Appelstr.~11, 30167, Hannover, Germany (%
\texttt{marino@ikm.uni-hannover.de}).} \footnotetext[3]{%
Universit\`{a} degli Studi di Pavia, Dipartimento di Matematica, and IMATI-C.N.R., Via Ferrata
1, 27100, Pavia, Italy (\texttt{elisabetta.rocca@unipv.it}).}

\section{Introduction}

Collisions of solids produce discontinuities of velocities and
discontinuities of temperatures at collision time $t$. We consider a solid
made of shape memory alloys. It occupies domain $\Omega $ with boundary $%
\partial \Omega $. There is a vast literature on shape memory alloys. We
mention only the predictive theory which introduces besides the macroscopic
velocity and the temperature, velocities at the microscopic level which are
responsible for the phase changes between the martensites and austenite
phases, \cite{fre1987}, \cite{fremya1996}. We have chosen to represent at
the macroscopic level, the velocities at the microscopic level by the
velocities
\begin{equation}
\frac{d\beta _{i}}{dt},  \notag
\end{equation}%
of the volumes fractions $\beta _{i}$ of the different phases, $\beta _{3}$
for the austenite phase and $\beta _{1}$, $\beta _{2}$ for the martensite
phases, assuming there are two of them. There can be voids in the mixture of
the three phases with volume fraction $\beta_{voids}$.

We investigate collisions involving shape memory alloys and assume these
collisions are instantaneous, \cite{fre2007}. We denote with subscript $^{-}$%
, quantities before collision and subscript $^{+}$, quantities after
collision. For example, we denote
\begin{equation}
\vec{U}=(\vec{U}^{+},\vec{U}^{-}),  \notag
\end{equation}%
the actual velocities, $\vec{U}^{-}$ being the actual velocity before the
collision and $\vec{U}^{+}$ being the actual velocity after. We denote $%
\left[ X\right] =X^{+}-X^{+}$, the discontinuity of quantity $X$

In collisions, there are rapid variations of the velocities at the
microscopic level resulting in rapid variations of the volumes fractions $%
\beta _{i}$. Thus we assume also the volume fractions are discontinuous \cite%
{fre2007}, \cite{fremond2011}, \cite{fremar2012}, \cite{mar2013}
\begin{equation}
\left[ \beta _{i}\right] =\beta _{i}^{+}-\beta _{i}^{-}.  \notag
\end{equation}

The collisions being dissipative phenomena, they produce burst of heat which
intervene in the thermal evolution. They result in temperature
discontinuities and they may produce phase changes. Moreover voids may also
appear, \cite{freroc2009}. Thus the volume fractions discontinuities and the
temperature discontinuities are coupled as they are in smooth evolutions,
\cite{fre2000}. Transient and fast but smooth phenomena in shape memory
alloys are investigated in \cite{chen2000}, \cite{col2009}. The last paper
contains experimental results.

\section{The Model}

\subsection{The State Quantities and the Quantities which Describe the
Evolution}

The state quantities are
\begin{equation}
E=\left( \varepsilon (\vec{u}),\beta _{i},\mathbf{grad}\beta _{i},T\right) .
\notag
\end{equation}%
In a collision, the small displacement $\vec{u}$ does not change, thus the
small deformation, $\varepsilon (\vec{u})$, remains constant. But as already
seen, the phase volume fractions and the temperature $T$ do vary in
collisions. We have the state quantities before collision, $E^{-}$, and $%
E^{+}$ after.

The quantities which describe the evolution are the evolution of the
velocity of deformation
\begin{equation*}
D(\frac{\vec{U}^{+}+\vec{U}^{-}}{2}),
\end{equation*}%
where $D$ is the usual deformation operator, and the gradient of the average
temperature introduced in collision theory, together with the variation of
the volume fractions and their gradients. The average temperature is%
\begin{equation}
\underline{T}=\frac{T^{+}+T^{-}}{2},  \notag
\end{equation}%
where $T^{+}$ and $T^{-}$ are the temperatures after and before collision.
Its gradient is involved in the description of heat diffusion occurring in
collisions. Thus the quantities which describe the evolution, $\delta E$ is%
\begin{equation}
\delta E=(D(\frac{\vec{U}^{+}+\vec{U}^{-}}{2}),\left[ \beta _{i}\right] ,%
\mathbf{grad}\left[ \beta _{i}\right] ,\mathbf{grad}\underline{T}).  \notag
\end{equation}%
The discontinuity $\left[ \beta _{i}\right] $ is the non smooth part of the
velocity $d\beta _{i}/dt$. Let us note that $\delta E$ is objective.

\subsection{The Equations of Motion}

The equations of motion result from the principle of virtual work
introducing percussion stress $\Sigma $, percussion work $B_{p}$ and
percussion work flux vector $\vec{H}_{p}$, \cite{fremar2016}. They are
\begin{gather}
\rho \left[ \vec{U}\right] =\func{div}\Sigma ,\ in\;\Omega ,  \label{mouv1}
\\
\Sigma =\Sigma ^{T},\ in\;\Omega ,  \label{mouv1bis} \\
-B_{i}^{p}+\func{div}\vec{H}_{i}^{p}=0,\ in\;\Omega ,  \label{mouv2}
\end{gather}
and
\begin{equation}
\Sigma \vec{N}=\vec{G}^{p},\ \vec{H}_{i}^{p}\cdot \vec{N}=0,\ on\ \partial
\Omega ,  \label{mouvcl}
\end{equation}%
where $\vec{G}^{p}$ is the surface external percussion, \cite{fremar2016}.
It is assumed no body percussion and no external action at the microscopic
level. Equation (\ref{mouv1bis}) is the angular momentum equation of motion,
see \cite{frevirtwork2017} p. 248-251.

\subsection{The Mass Balance}

It is%
\begin{equation}
\left[ \rho (\beta _{1}+\beta _{2}+\beta _{3})\right] =0,  \notag
\end{equation}%
or%
\begin{equation}
\left[ (\beta _{1}+\beta _{2}+\beta _{3})\right] =0,  \label{mass}
\end{equation}%
assuming the density $\rho $ is constant and the same for each phase. A possible
evolution of the time discontinuity of the voids volume fraction
\begin{equation}
\beta _{voids}=1-(\beta _{1}+\beta _{2}+\beta _{3}),  \notag
\end{equation}
is given by a constitutive law. Examples are given in \cite{fremar2016}.

\subsection{The First Law of Thermodynamics}

The first and second laws of thermodynamics intervene in the derivation of
the constitutive laws. We recall them to get the new mechanical and thermal
collision constitutive laws.

The first law can be written as
\begin{equation}
\left[ \mathcal{E}\right] +\left[ \mathcal{K}\right] =\mathcal{T}_{ext}%
+\mathcal{C},  \label{9energie}
\end{equation}%
where
\begin{equation}
\mathcal{E=}\int_{\Omega }ed\Omega ,  \notag
\end{equation}%
is the internal energy, $\mathcal{K}$ is the kinetic energy, $\mathcal{C}$
is the thermal impulse received by the solid, and $\mathcal{T}_{ext}$ is
actual work of the external forces. With the principle of virtual work where
the velocities are the actual velocities, i.e., with the theorem of kinetic
energy, the first law gives
\begin{equation}
\left[ \mathcal{E}\right] =-\mathcal{T}_{int}+\mathcal{C},  \notag
\end{equation}%
where $\mathcal{T}_{int}$ is actual work of the internal forces. The
temperature may be discontinuous, \cite{fre2007}, \cite{fre2000}: we have
already defined $T^{-}$ the temperature before the collision and $T^{+}$ the
temperature after the collision and
\begin{equation}
\underline{T}=\frac{T^{+}+T^{-}}{2}.  \notag
\end{equation}%
We assume that the external impulse heat is received either at temperature $%
T^{-}$ or at temperature $T^{+}$
\begin{equation}
\mathcal{C=}\int_{\partial \Omega }-\left( T^{+}\vec{Q}_{p}^{+}+T^{-}\vec{Q}%
_{p}^{-}\right) \cdot \vec{N}d\Gamma +\int_{\Omega }T^{+}\mathcal{B}%
^{+}+T^{-}\mathcal{B}^{-}d\Omega ,  \notag
\end{equation}%
where $\vec{Q}_{p}$ is the impulsive entropy flux vector and $\mathcal{B}$
the impulsive entropy source. Relationship (\ref{9energie}) being true for
any subdomain of $\Omega $, we get the energy balance law
\begin{eqnarray}
\left[ e\right] &=&\Sigma :D(\frac{\vec{U}^{+}+\vec{U}^{-}}{2})+B_{i}^{p}%
\left[ \beta _{i}\right] +\vec{H}_{i}^{p}\cdot \mathbf{grad}\left[ \beta _{i}%
\right] \\
&&-\func{div}(T^{+}\vec{Q}_{p}^{+}+T^{-}\vec{Q}_{p}^{-})+T^{+}\mathcal{B}%
^{+}+T^{-}\mathcal{B}^{-}.  \notag
\end{eqnarray}%
By using the Helmholtz relationship, $e=\Psi +Ts$, we have
\begin{gather}
\left[ e\right] =\left[ \Psi \right] +\underline{s}\left[ T\right] +%
\underline{T}\left[ s\right]  \notag \\
=\Sigma :D(\frac{\vec{U}^{+}+\vec{U}^{-}}{2})+B_{i}^{p}\left[ \beta _{i}%
\right] +\vec{H}_{i}^{p}\cdot \mathbf{grad}\left[ \beta _{i}\right]  \notag
\\
-\func{div}\left\{ \underline{T}\Sigma \left( \vec{Q}_{p}\right) +\left[ T%
\right] \Delta \left( \vec{Q}_{p}\right) \right\} +\underline{T}\Sigma
\left( \mathcal{B}\right) +\left[ T\right] \Delta \left( \mathcal{B}\right) ,
\label{9relat1}
\end{gather}%
where a sum
\begin{equation}
T^{+}\mathcal{B}^{+}+T^{-}\mathcal{B}^{-}=\Sigma (T\mathcal{B}),  \notag
\end{equation}%
is split in an other sum
\begin{equation}
\Sigma (T\mathcal{B})=\underline{T}\Sigma \left( \mathcal{B}\right) +\left[ T%
\right] \Delta \left( \mathcal{B}\right) ,  \notag
\end{equation}%
with
\begin{equation}
\Sigma \left( \mathcal{B}\right) =\mathcal{B}^{+}+\mathcal{B}^{-}\ and,\
\Delta \left( \mathcal{B}\right) =\frac{\mathcal{B}^{+}-\mathcal{B}^{-}}{2}.
\notag
\end{equation}

\begin{remark}
To avoid too many notation, we use letter $\Sigma $ with two meanings: the
percussion stress $\Sigma $ which appears in the equation of motion and the
sum $\Sigma \left( \mathcal{B}\right) =\mathcal{B}^{+}+\mathcal{B}^{-}$.
They appear in different context and the sum $\Sigma \left( \mathcal{B}%
\right) $ has always an argument.
\end{remark}

\subsection{The Second Law of Thermodynamics}

The Second Law may be stated as
\begin{equation}
\left[ \mathcal{S}\right] =\int_{\Omega }\left[ s\right] d\Omega \geq
-\int_{\Gamma }\left( \vec{Q}_{p}^{+}+\vec{Q}_{p}^{-}\right) \cdot \vec{N}%
d\Gamma +\int_{\Omega }\mathcal{B}^{+}+\mathcal{B}^{-}d\Omega ,  \notag
\end{equation}%
which gives%
\begin{equation}
\left[ s\right] \geq -\func{div}\Sigma \left( \vec{Q}_{p}\right) +\Sigma
\left( \mathcal{B}\right) .  \label{relat2}
\end{equation}

Combining relationships (\ref{9relat1}) and (\ref{relat2}), we get%
\begin{gather}
\left[ \Psi \right] +\underline{s}\left[ T\right] +\func{div}\left( \left[ T%
\right] \Delta (\vec{Q}_{p})\right) -\left[ T\right] \Delta (\mathcal{B)}
\notag \\
\leq \Sigma :D(\frac{\vec{U}^{+}+\vec{U}^{-}}{2})+B_{i}^{p}\left[ \beta _{i}%
\right] +\vec{H}_{i}^{p}\cdot \mathbf{grad}\left[ \beta _{i}\right] -\mathbf{%
grad}\underline{T}\cdot \Sigma (\vec{Q}_{p}).  \label{relation3}
\end{gather}%
Let us note that the right hand side of (\ref{relation3}) is a scalar
product between internal forces and related evolution quantities whereas the
left hand side is not a scalar product. Let us try to relate $\left[ \Psi %
\right] $ to a scalar product. We have
\begin{gather*}
\left[ \Psi \right] =\Psi (T^{+},\beta _{i}^{+},\mathbf{grad}\beta
_{i}^{+})-\Psi (T^{-},\beta _{i}^{-},\mathbf{grad}\beta _{i}^{-}) \\
=\Psi (T^{+},\beta ^{+},\mathbf{grad}\beta ^{+})-\Psi (T^{+},\beta _{i}^{-},%
\mathbf{grad}\beta _{i}^{-}) \\
+\Psi (T^{+},\beta _{i}^{-},\mathbf{grad}\beta _{i}^{-})-\Psi (T^{-},\beta
_{i}^{-},\mathbf{grad}\beta _{i}^{-}).
\end{gather*}

Because the free energy is a concave function of temperature $T$, we have%
\begin{equation*}
\Psi (T^{+},\beta _{i}^{-},\mathbf{grad}\beta _{i}^{-})-\Psi (T^{-},\beta
_{i}^{-},\mathbf{grad}\beta _{i}^{-})\leq -s^{fe}\left[ T\right] ,
\end{equation*}%
with%
\begin{equation}
-s^{fe}\in \hat{\partial}\Psi _{T}(T^{-},\beta _{i}^{-},\mathbf{grad}\beta
_{i}^{-}),  \label{9.entropie}
\end{equation}%
where $\hat{\partial}\Psi _{T}$ is the set of the uppergradients of the
concave function
\begin{equation*}
T\rightarrow \Psi _{T}(T,\beta _{i}^{-},\mathbf{grad}\beta _{i}^{-})=\Psi
(T,\beta _{i}^{-},\mathbf{grad}\beta _{i}^{-}).
\end{equation*}%
We assume $\Psi $ is a convex function of $(\vec{\beta},\mathbf{grad}$ $\vec{%
\beta})$. Thus we have%
\begin{equation*}
\Psi (T^{+},\beta _{i}^{+},\mathbf{grad}\beta _{i}^{+})-\Psi (T^{+},\beta
_{i}^{-},\mathbf{grad}\beta _{i}^{-})\leq B_{i}^{fe}\left[ \beta _{i}\right]
+\vec{H}_{i}^{fe}\cdot \mathbf{grad}\left[ \beta _{i}\right] ,
\end{equation*}%
with%
\begin{equation}
\left( B_{i}^{fe},\vec{H}_{i}^{fe}\right) \in \partial \Psi _{\beta ,\mathbf{%
grad}\beta }(T^{+},\beta _{i}^{+},\mathbf{grad}\beta _{i}^{+}),
\label{relation6}
\end{equation}%
where $\partial \Psi _{\beta ,\mathbf{grad}\beta }$ is the subdifferential
set of convex function $\Psi $ of $(\vec{\beta},\mathbf{grad}$ $\vec{\beta})$%
. The internal forces $\left( B_{i}^{fe},\vec{H}_{i}^{fe}\right) $ depend on
the future state $(T^{+},\beta _{i}^{+},\mathbf{grad}$ $\beta _{i}^{+})$, in
agreement with our idea that the constitutive laws sum up what occurs during
the discontinuity of the state quantities. It results%
\begin{eqnarray*}
&&\left[ \Psi \right] +\underline{s}\left[ T\right] +\hbox{div}\left[ T%
\right] \Delta (\vec{Q}_{p})-\left[ T\right] \Delta (\mathcal{B)} \\
&\leq &B_{i}^{fe}\left[ \beta _{i}\right] +\vec{H}_{i}^{fe}\cdot \mathbf{grad%
}\left[ \beta _{i}\right] -s^{fe}\left[ T\right] +\bar{s}\left[ T\right] +%
\hbox{div}\left[ T\right] \Delta (\vec{Q}_{p}) \\
&=&B_{i}^{fe}\left[ \beta _{i}\right] +\vec{H}_{i}^{fe}\cdot \mathbf{grad}%
\left[ \beta _{i}\right] +(-s^{fe}+\bar{s}+\hbox{div}\Delta (\vec{Q}_{p}))%
\left[ T\right] +\Delta (\vec{Q}_{p})\cdot \mathbf{grad}\left[ T\right] ..
\end{eqnarray*}

As usual, we assume no dissipation with respect to $\left[ T\right] $, \cite%
{fre2007}, \cite{fremond2011}, \cite{fre2000} and have%
\begin{gather}
\Delta (\vec{Q}_{p})=0,  \label{9.nondissipation} \\
-s^{fe}+\bar{s}+\func{div}\Delta (\vec{Q}_{p})-\Delta (\mathcal{B)}=0.
\notag
\end{gather}%
Thus%
\begin{equation}
\Delta (\mathcal{B)+}S=0,  \label{relation4}
\end{equation}%
with $S=s^{fe}-\bar{s}$. This relationship splits either the received heat
impulse, $T^{+}\mathcal{B}^{+}+T^{-}\mathcal{B}^{-}$, or the received
entropy impulse, $\mathcal{B}^{+}+\mathcal{B}^{-}$, between the two
temperatures, $T^{+}$ and $T^{-}$. Let us note that this relationship
depends on the future state via the average entropy $\bar{s}$.

We may choose a pseudo-potential of dissipation%
\begin{equation*}
\Phi (\delta E^{\pm },E^{+},E^{-})=\Phi (D(\frac{\vec{U}^{+}+\vec{U}^{-}}{2}%
),\left[ \beta _{i}\right] ,\mathbf{grad}\left[ \beta _{i}\right] ,\mathbf{%
grad}\underline{T},\underline{T}),
\end{equation*}%
and constitutive laws%
\begin{eqnarray}
&&\left( \Sigma ,(B_{i}^{p}-B_{i}^{fe}),(\vec{H}_{i}^{p}-\vec{H}_{i}^{fe}),-2%
\vec{Q}_{p}\right)  \notag \\
&\in &\partial \Phi (D(\frac{\vec{U}^{+}+\vec{U}^{-}}{2}),\left[ \beta _{i}%
\right] ,\mathbf{grad}\left[ \beta _{i}\right] ,\mathbf{grad}\underline{T},%
\underline{T}).  \label{relation7}
\end{eqnarray}

It results from this choice that the internal forces satisfy inequality
\begin{equation}
0\leq \Sigma :D(\frac{\vec{U}^{+}+\vec{U}^{-}}{2})+(B_{i}^{p}-B_{i}^{fe})%
\left[ \beta _{i}\right] +(\vec{H}_{i}^{p}-\vec{H}_{i}^{fe})\cdot \mathbf{%
grad}\left[ \beta _{i}\right] -2\mathbf{grad}\underline{T}\cdot \vec{Q}_{p},
\label{relation5}
\end{equation}

and the second law is satisfied

\begin{theorem}
If constitutive laws (\ref{relation6}), (\ref{relation4}) and (\ref%
{relation7}) are satisfied, then the second law is satisfied.
\end{theorem}

\begin{proof}
If relationship (\ref{relation7}) is satisfied, inequality (\ref{relation5})
is satisfied. Then it is easy to prove that the inequality (\ref{relation3})
which is equivalent to the second law is satisfied.
\end{proof}

\begin{remark}
The discontinuity $\left[ \Psi \right] $ may be split in a different manner%
\begin{gather*}
\left[ \Psi \right] =\Psi (T^{+},\beta _{i}^{+},\mathbf{grad}\beta
_{i}^{+})-\Psi (T^{-},\beta _{i}^{-},\mathbf{grad}\beta _{i}^{-}) \\
=\Psi (T^{+},\beta _{i}^{+},\mathbf{grad}\beta _{i}^{+})-\Psi (T^{-},\beta
_{i}^{+},\mathbf{grad}\beta _{i}^{+}) \\
+\Psi (T^{-},\beta _{i}^{+},\mathbf{grad}\beta _{i}^{+})-\Psi (T^{-},\beta
_{i}^{-},\mathbf{grad}\beta _{i}^{-}).
\end{gather*}%
We get%
\begin{equation*}
\Psi (T^{+},\beta _{i}^{+},\mathbf{grad}\beta _{i}^{+})-\Psi (T^{-},\beta
_{i}^{+},\mathbf{grad}\beta _{i}^{+})\leq -\hat{s}^{fe}\left[ T\right] ,
\end{equation*}%
\begin{equation*}
-\hat{s}^{fe}\in \hat{\partial}\Psi _{T}(T^{-},\beta _{i}^{+},\mathbf{grad}%
\beta _{i}^{+}).
\end{equation*}%
If $\Psi $ it is a convex function of $(\beta ,\mathbf{grad}$ $\beta )$%
\begin{equation*}
\Psi (T^{-},\beta _{i}^{+},\mathbf{grad}\beta _{i}^{+})-\Psi (T^{-},\beta
_{i}^{-},\mathbf{grad}\beta _{i}^{-}))\leq \hat{B}^{fe}\left[ \beta \right] +%
\hat{H}^{fe}\cdot \mathbf{grad}\left[ \beta \right]
\end{equation*}%
\begin{equation*}
\left( \hat{B}^{fe},\hat{H}^{fe}\right) \in \partial \Psi _{\beta ,\mathbf{%
grad}\beta }(T^{-},\beta _{i}^{-},\mathbf{grad}\beta _{i}^{-}).
\end{equation*}%
The internal forces $\left( \hat{B}^{fe},\hat{H}^{fe}\right) $ depend
entirely on the past state $(T^{-},\beta _{i}^{-},\mathbf{grad}\beta
_{i}^{-})$. Because we think that the constitutive laws sum up what occurs
during the discontinuity, it is mandatory that the internal forces depend on
the future state $(T^{+},\beta _{i}^{+},\mathbf{grad}$ $\beta _{i}^{+})$.
Thus this splitting of the free energy does not seem as good as the one we
have chosen.
\end{remark}

\subsection{The Free Energy}

\label{sec:9.7}

A shape memory alloy is considered as a mixture of the martensite and
austenite phases with volume fractions $\beta _{i}$. The volume free energy
of the mixture we choose is 
\begin{equation}
\Psi =\Psi (E)=\sum\limits_{i=1}^{3}\beta _{i}\Psi _{i}(E)+h(E),  \label{10supprimer}
\end{equation}
where the $\Psi _{i}$'s are the volume free energies of the ${i}$ phases and
$h$ is a free energy describing interactions between the different phases.
We have assumed that internal constraints are physical properties, hence, we
decide to choose properly the two functions describing the material, i.e.,
the free energy $\Psi $ and the pseudo-potential of dissipation $\Phi $, in
order to take these constraints into account. Since, the pseudo-potential
describes the kinematic properties (i.e., properties which depend on the
velocities) and the free energy describes the state properties, obviously
the internal constraints
\begin{equation}
0\leq \beta _{i}\leq 1,  \label{1vide}
\end{equation}%
and
\begin{equation}
\beta _{1}+\beta _{2}+\beta _{3}\leq 1,  \label{2vide}
\end{equation}%
because voids may appear, are to be taken into account with the free energy $%
\Psi $, \cite{freroc2009}.

For this purpose, we assume the $\Psi _{i}$'s are defined over the whole
linear space spanned by $\beta _{i}$ and the free energy is defined by
\begin{equation*}
\Psi (E)=\beta _{1}\Psi _{1}(E)+\beta _{2}\Psi _{2}(E)+\beta _{3}\Psi
_{3}(E)+h(E)\,.
\end{equation*}%
We choose the very simple interaction free energy%
\begin{equation*}
h(E)=I_{C}(\vec{\beta})+\frac{k}{2}\left\vert \mathbf{grad}\vec{\beta}%
\right\vert ^{2},
\end{equation*}
where $\vec{\beta}=(\beta_1,\beta_2,\beta_3)$ and $I_{C}$ is
the indicator function of the convex set
\begin{equation}
C=\{(\gamma_{1},\gamma_{2},\gamma_{3})\in \mathbb{R}^{3};0\leq
\gamma_{i}\leq 1;\gamma_{1}+\gamma_{2}+\gamma_{3}\leq 1\}\,.  \label{defC}
\end{equation}
Moreover, by $(k/2)\left\vert \mathbf{grad}\vec{\beta}%
\right\vert ^{2}$ we mean the product of two tensors $\mathbf{grad}\vec{\beta%
}$ multiplied by the \textsl{interfacial energy }coefficient $(k/2)>0$. The
terms $I_{C}(\vec{\beta})+(k/2)\left\vert \mathbf{grad}\vec{\beta}%
\right\vert ^{2}$ may be seen as a \textsl{mixture or interaction free-energy%
}.

The only effect of $I_{C}(\vec{\beta})$ is to guarantee that the proportions
$\beta _{1}$, $\beta _{2}$ and $\beta _{3}$ take admissible physical values,
i.e.~they satisfy constraints (\ref{1vide}) and (\ref{2vide}) (see also \ref%
{defC}). The interaction free energy term $I_{C}(\vec{\beta})$ is equal to
zero when the mixture is physically possible ($\vec{\beta}\in C$) and to $%
+\infty $ when the mixture is physically impossible ($\vec{\beta}\notin C$).

Let us note even if the free energy of the voids phase is $0$, the voids
phase has physical properties due to the interaction free energy term $%
(k/2)\left\vert \mathbf{grad}\vec{\beta}\right\vert ^{2}$ which depends on
the gradient of $\vec{\beta}$. It is known that this gradient is related to
the interfaces properties: $\mathbf{grad}{\beta _{1}}$, $\mathbf{grad}{\beta
_{2}}$ describes properties of the voids-martensites interfaces and $\mathbf{%
grad}{\beta _{3}}$ describes properties of the voids-austenite interface. In
this setting, the voids have a role in the phase change and make it
different from a phase change without voids. The model is simple and
schematic but it may be upgraded by introducing sophisticated interaction
free energy depending on $\vec{\beta}$ and on $\mathbf{grad}\vec{\beta}$.

\begin{remark}
A slightly more sophisticated interfacial energy is
\begin{equation}
h(E)=I_{C}(\vec{\beta})+\frac{1}{2}(k_{i}(\mathbf{grad}{\beta _{i})}^{2}),
\label{dif1}
\end{equation}%
with different interaction phase parameters $k_{i}$.
\end{remark}

For the sake of simplicity, we choose the volume free energies, \cite%
{fre1987}, \cite{fremya1996}, \cite{fremond2011}
\begin{gather*}
\Psi _{1}(E)=\frac{1}{2}\varepsilon (\vec{u}):K:\varepsilon (\vec{u})-\tau
(T)\mathbf{1}:\varepsilon (\vec{u})-CT\log T, \\
\Psi _{2}(E)=\frac{1}{2}\varepsilon (\vec{u}):K:\varepsilon (\vec{u})+\tau
(T)\mathbf{1}:\varepsilon (\vec{u})-CT\log T, \\
\Psi _{3}(E)=\frac{1}{2}\varepsilon (\vec{u}):K:\varepsilon (\vec{u})-\frac{%
l_{a}}{T_{0}}(T-T_{0})-CT\log T,
\end{gather*}%
where $K$ is the volume elastic tensor and $C$ the volume heat capacities of
the phases and the quantity $l_{a}$ is the latent heat martensite-austenite
volume phase change at temperature $T_{0}$, $\mathbf{1}$ is the unit tensor.

Concerning the function $\tau (T)$, we assume the schematic simple
expression
\begin{equation*}
\tau (T)=(T-T_{c})\overline{\tau },\;{for\ }T\leq T_{c},\tau (T)=0,\;{for\ }%
T\geq T_{c},
\end{equation*}%
with $\overline{\tau }\leq 0$ and assume the temperature $T_{c}$ is greater
than $T_{0}$. With those assumptions, it results
\begin{gather*}
\Psi (E)=\frac{\beta _{1}+\beta _{2}+\beta _{3}}{2}\left\{ \varepsilon (\vec{%
u}):K:\varepsilon (\vec{u})\right\} \\
-(\beta _{1}-\beta _{2})\tau (T)\mathbf{I}:\varepsilon (\vec{u})-\beta _{3}%
\frac{l_{a}}{T_{0}}(T-T_{0}) \\
-(\beta _{1}+\beta _{2}+\beta _{3})CT\log T+\frac{k}{2}\left\vert \mathbf{%
grad}\vec{\beta}\right\vert ^{2}+I_{C}(\vec{\beta})\,.
\end{gather*}

\begin{remark}
Depending on the sign of $\mathbf{I}:\varepsilon (\vec{u})=\func{div}\vec{u}$%
, free energy $\Psi $ is either a concave or a convex function of
temperature $T$. As explained in \cite{fremond2011} (see Remark 5.3 page
72), it is easy to overcome this difficulty to have in any case $\Psi $ a
concave function of $T$. Experiments show that rigidity matrix $K$ depends
on $T$. With this result, it is easy to have $\Psi $ a concave function of $%
T $, \cite{fremond2011}. In this presentation, we keep the schematic
expression for $\Psi $ and we note that we will assume the solid is not
deformed when colliding, i.e., $\func{div}\vec{u}=0$. In this situation, the
schematic free energy is a concave function of $T$.
\end{remark}

\subsection{The Pseudo-potential of Dissipation}

From experiments, it is known that the behaviour of shape memory alloys
depends on time, i.e., the behaviour is dissipative. We define a
pseudo-potential of dissipation with 
\begin{gather}
\Phi ((D(\frac{\vec{U}^{+}+\vec{U}^{-}}{2}),\left[ \beta _{i}\right] ,%
\mathbf{grad}\left[ \beta _{i}\right] ,\mathbf{grad}\underline{T},\underline{%
T})) \notag \\
=k_{v}\left( D(\frac{\vec{U}^{+}+\vec{U}^{-}}{2})\right) ^{2}+\frac{c}{2}
\sum_{i=1}^{3}\left( \left[ \beta _{i}\right] \right) ^{2}+\frac{\upsilon }{2%
}\sum_{i=1}^{3}\left( \mathbf{grad}\left[ \beta _{i}\right] \right) ^{2}+%
\frac{\lambda }{2\underline{T}}\left( \mathbf{grad}\underline{T}\right) ^{2}
\notag \\
+I_{0}(\left[ \beta _{1}+\beta _{2}+\beta _{3}\right] ), \label{pseudo}
\end{gather}
where $\lambda \geq 0$ represents the thermal conductivity and $k_{v}>0$, $%
c\geq 0$, $\upsilon \geq 0$ stand for collisions viscosities related to
macroscopic and microscopic dissipative phenomena.

\begin{remark}
A slightly more sophisticated pseudo-potential of dissipation is
\begin{gather}
\Phi ((D(\frac{\vec{U}^{+}+\vec{U}^{-}}{2}),\left[ \beta _{i}\right] ,%
\mathbf{grad}\left[ \beta _{i}\right] ,\mathbf{grad}\underline{T},\underline{%
T}))  \label{dif2} \\
=k_{v}\left( D(\frac{\vec{U}^{+}+\vec{U}^{-}}{2})\right) ^{2}+\frac{1}{2}%
\sum_{i=1}^{3}c_{i}\left( \left[ \beta _{i}\right] \right) ^{2}+\frac{1}{2}%
\sum_{i=1}^{3}\frac{\upsilon _{i}}{2}\left( \mathbf{grad}\left[ \beta _{i}%
\right] \right) ^{2}+\frac{\lambda }{2\underline{T}}\left( \mathbf{grad}%
\underline{T}\right) ^{2}  \notag \\
+I_{0}(\left[ \beta _{1}+\beta _{2}+\beta _{3}\right] ),  \label{dif2}
\end{gather}
involving different viscosities for the phases $c_{i}\geq 0$, $\upsilon
_{i}\geq 0$.
\end{remark}

The pseudo-potential takes into account the mass balance relationship (i.e.,
Eq. (\ref{mass})), $I_{0}$ being the indicator function of the origin of $%
\mathbb{R}$.

\subsection{The Constitutive Laws}

They are given by relationships (\ref{9.nondissipation}), (\ref{relation4})
and (\ref{relation6})%
\begin{equation*}
\left( B_{i}^{fe},\vec{H}_{i}^{fe}\right) \in \partial \Psi _{\beta ,\mathbf{%
grad}\beta }(T^{+},\beta _{i}^{+},\mathbf{grad}\beta _{i}^{+}),
\end{equation*}%
giving

\begin{gather*}
\vec{B}^{fe}=\left\vert
\begin{array}{c}
\frac{1}{2}\varepsilon (\vec{u}):K:\varepsilon (\vec{u})-\tau (T^{+})\mathbf{%
1}:\varepsilon (\vec{u})-CT^{+}\log T^{+} \\
\frac{1}{2}\varepsilon (\vec{u}):K:\varepsilon (\vec{u})+\tau (T^{+})\mathbf{%
1}:\varepsilon (\vec{u})-CT^{+}\log T^{+} \\
\frac{1}{2}\varepsilon (\vec{u}):K:\varepsilon (\vec{u})-\frac{l_{a}}{T_{0}}%
(T^{+}-T_{0})-CT^{+}\log T^{+}%
\end{array}%
\right\vert +\vec{B}_{reca}^{fe}, \\
\vec{B}_{reca}^{fe}\in \partial I_{C}(\vec{\beta}^{+}), \\
\vec{H}_{i}^{fe}=k\mathbf{grad}\beta _{i}^{+},
\end{gather*}

and by relationship (\ref{relation7})%
\begin{gather*}
\Sigma =k_{v}D(\vec{U}^{+}+\vec{U}^{-}), \\
B_{i}^{p}-B_{i}^{fe}=c\left[ \beta _{i}\right] -P, \\
-P\in \partial I_{0}(\left[ \beta _{1}+\beta _{2}+\beta _{3}\right] )=%
\mathbb{R}, \\
\vec{H}_{i}^{p}-\vec{H}_{i}^{fe}=\upsilon \mathbf{grad}\left[ \beta _{i}%
\right] , \\
-2\vec{Q}_{p}=\frac{\lambda }{\underline{T}}\mathbf{grad}\underline{T},
\end{gather*}%
where $P$ is the percussion reaction pressure due to the mass balance, and
\begin{gather*}
e=(\beta _{1}+\beta _{2}+\beta _{3})CT+l_{a}\beta _{3}+\frac{1}{2}%
\varepsilon (\vec{u}):K:\varepsilon (\vec{u})+\frac{k}{2}\left\vert \mathbf{%
grad}\vec{\beta}\right\vert ^{2} \\
-(\beta _{1}-\beta _{2})(\tau (T)-T\partial \tau (T))\mathbf{I}:\varepsilon (%
\vec{u}).
\end{gather*}

The internal energy within the small perturbation assumption is
\begin{equation*}
e=(\beta _{1}+\beta _{2}+\beta _{3})CT+l_{a}\beta _{3}.
\end{equation*}

\subsection{The Equations in a Collision}

They result from the energy balance, the equations of motion, the
constitutive laws and the initial situation, i.e., the situation before the
collision.

\paragraph{The Energy Balance}

We assume adiabatic evolution%
\begin{equation*}
T^{+}\mathcal{B}^{+}+T^{-}\mathcal{B}^{-}=0.
\end{equation*}

The energy balance results from relationships (\ref{9energie}), (\ref%
{relation6}) and (\ref{9.nondissipation})
\begin{equation*}
\left[ e\right] +\func{div}(2\underline{T}\vec{Q}_{p})=\Sigma :D(\frac{\vec{U%
}^{+}+\vec{U}^{-}}{2})+B_{i}^{p}\left[ \beta _{i}\right] +\vec{H}%
_{i}^{p}\cdot \mathbf{grad}\left[ \beta _{i}\right] .
\end{equation*}

Note that the reactions, for instance $\vec{B}_{reca}^{fe}$, work, with work%
\begin{equation*}
\vec{B}_{reca}^{fe}\left[ \vec{\beta}\right] \geq 0.
\end{equation*}%
This a property of collisions: the reactions to perfect constraints work
whereas they do not work in smooth evolutions, \cite{fre2007}.

\paragraph{The Equations of Motion and the Mass Balance}

The equations of motion are (\ref{mouv1}), (\ref{mouv1bis}), (\ref{mouv2})
and (\ref{mouvcl}). For the sake of simplicity, we assume the
voids volume fraction, \cite{rod2000}, \cite{freroc2009},
\begin{equation}
\beta _{voids}=1-(\beta _{1}+\beta _{2}+\beta _{3}),  \notag
\end{equation}%
is null and does not evolve in the collision, \cite{rod2000}, \cite%
{freroc2009}. Assuming no interpenetration before collision%
\begin{equation*}
\beta _{1}^{-}+\beta _{2}^{-}+\beta _{3}^{-}=1,
\end{equation*}%
the mass balance relationship (\ref{mass}) gives
\begin{equation}
\beta _{1}^{+}+\beta _{2}^{+}+\beta _{3}^{+}=1,  \label{1bis}
\end{equation}
and
\begin{equation}  \label{massbal}
(\beta _{1}+\beta _{2}+\beta _{3})^+=(\beta _{1}+\beta _{2}+\beta _{3})^-=1.
\end{equation}

\paragraph{The Constitutive Laws}

For the sake of simplicity, we assume the material is undeformed $%
\varepsilon (\vec{u})=0$ at collision time. Thus, using also \eqref{massbal}%
, we have
\begin{gather*}
\vec{B}^{fe}=\left\vert
\begin{array}{c}
0 \\
0 \\
-\frac{l_{a}}{T_{0}}(T^{+}-T_{0})%
\end{array}%
\right\vert -\left\vert
\begin{array}{c}
C T^+ \text{Log}T^+ \\
C T^+ \text{Log}T^+ \\
C T^+ \text{Log}T^+ \\
\end{array}%
\right\vert +\vec{B}_{reca}^{fe}, \\
\vec{B}_{reca}^{fe}\in \partial I_{C}(\vec{\beta}^{+}), \\
\vec{H}_{i}^{fe}=k\mathbf{grad}\beta _{i}^{+},
\end{gather*}
and by relationship (\ref{relation7})
\begin{gather*}
\Sigma =k_{v}D(\vec{U}^{+}+\vec{U}^{-}), \\
B_{i}^{p}-B_{i}^{fe}=c\left[ \beta _{i}\right]-P, \\
\vec{H}_{i}^{p}-\vec{H}_{i}^{fe}=\upsilon \mathbf{grad}\left[ \beta _{i}%
\right] , \\
-2\vec{Q}_{p}=\frac{\lambda }{\underline{T}}\mathbf{grad}\underline{T}.
\end{gather*}

It results internal force $\vec{B}^p$ is

\begin{gather*}
\vec{B}^p= c\left[ \vec{\beta}\right]+ \left\vert
\begin{array}{c}
0 \\
0 \\
-\frac{l_{a}}{T_{0}}(T^{+}-T_{0})%
\end{array}%
\right\vert -(C T^+ \text{Log}T^++P) \left\vert
\begin{array}{c}
1 \\
1 \\
1 \\
\end{array}%
\right\vert +\vec{B}_{reca}^{fe},
\end{gather*}
and by letting
\begin{equation}
\hat{P}=C T^+ \text{Log}T^++P,
\end{equation}
with
\begin{equation}
-\hat{P}\in \partial I_{0}(\left[ \beta _{1}+\beta _{2}+\beta _{3}\right] )=%
\mathbb{R},
\end{equation}
\begin{gather*}
\vec{B}^p= c\left[ \vec{\beta}\right]+ \left\vert
\begin{array}{c}
-\hat{P} \\
-\hat{P} \\
-\frac{l_{a}}{T_{0}}(T^{+}-T_{0})-\hat{P}%
\end{array}%
\right\vert +\vec{B}_{reca}^{fe},
\end{gather*}

The angular momentum equation of motion (\ref{mouv1bis}) is satisfied
and there are four equations for the unknowns $\vec{U}^{+}$, $\vec{\beta}%
^{+}$, $\hat{P}$ and $T^{+}$: the equations of motion for $\vec{U}^{+}$ and $%
\vec{\beta}^{+}$, the mass balance and the equation of motion related to $%
\vec{\beta}^{+}$ for the percussion $\hat{P}$ and the energy balance for the
temperature $T^{+}$.

\subsection{The Mechanical Equations}

As already said, there is an external surface percussion $\vec{G}^{p}$, for
instance an hammer stroke on part $\Gamma _{1}$ of the boundary, the solid
being fixed to a support on part $\Gamma _{0}$, with $\Gamma _{0},\Gamma
_{1} $ a partition of boundary $\delta \Omega $. We assume the solid is at
rest before collision and its temperature is uniform
\begin{equation*}
\vec{U}^{-}=0,\ \func{grad}T^{-}=0.
\end{equation*}

The equations $\vec{U}^{+}$, $\vec{\beta}^{+}$ and $\hat{P}$ are
\begin{gather}
\rho \vec{U}^{+}-k_{v}\func{div}D(\vec{U}^{+})=0,\ in\;\Omega ,  \label{equU}
\\
c\left[ \vec{\beta}\right] -\upsilon \Delta \left[ \vec{\beta}\right]
-k\Delta \vec{\beta}^{+}+\vec{B}_{reca}^{fe}+\left\vert
\begin{array}{c}
-\hat{P} \\
-\hat{P} \\
-\frac{l_{a}}{T_{0}}(T^{+}-T_{0})-\hat{P}%
\end{array}%
\right\vert =0,\ in\;\Omega ,  \label{equatemp} \\
\vec{B}_{reca}^{fe}\in \partial I_{C}(\vec{\beta}^{+}),  \label{convex} \\
\left[ \beta _{1}+\beta _{2}+\beta _{3}\right] =1.  \label{equaP}
\end{gather}

The boundary conditions are%
\begin{gather*}
\Sigma \vec{N}=\vec{G}^{p},\ \vec{H}_{i}^{p}\cdot \vec{N}=0,\ on\ \Gamma
_{1}, \\
\vec{U}^{+}=\vec{U}^{-}=0,\ \vec{H}_{i}^{p}\cdot \vec{N}=0,\ on\ \Gamma _{0}.
\end{gather*}%
where percussion $\vec{G}^{p}$ is the given hammer percussion on part $%
\Gamma _{1}$. The solid is fixed on an immobile support on part $\Gamma _{0}$%
.. Quantities $\beta ^{-}$ before collision are known.

\begin{remark}
We may note that any term of the free energy which depends on $\beta
_{1}+\beta _{2}+\beta _{3}$ does not intervene in the equation giving $\vec{%
\beta}^{+}$ because its derivatives with respect to the $\beta _{i}$ are
absorbed by the reaction percussion pressure $P$. This is the case of
quantity $C T^+ \text{Log}T^+$.
\end{remark}

\subsection{The Thermal Equation}

The equation for $T^{+}$ is using the mass balance (\eqref{massbal})
\begin{gather}
C\left[ T\right] +l_{a}\left[ \beta _{3}\right] +\frac{k}{2}\left[
\left\vert \mathbf{grad}\vec{\beta}\right\vert ^{2}\right] -\lambda \Delta
\underline{T}  \notag \\
=\Sigma :D(\frac{\vec{U}^{+}+\vec{U}^{-}}{2})+B_{i}^{p}\left[ \beta _{i}%
\right] +\vec{H}_{i}^{p}\cdot \mathbf{grad}\left[ \beta _{i}\right] ,\
in\;\Omega ,\   \label{equabeta}
\end{gather}%
with boundary condition%
\begin{equation}
\frac{\partial \underline{T}}{\partial N}=0,  \label{?????}
\end{equation}%
assuming no external heat impulse on part $\Gamma _{1}$ of the boundary and $%
\underline{T}$ or $T^{+}$ is given on part $\Gamma _{0}$. Note that another
boundary condition may be%
\begin{equation}
\lambda \frac{\partial \underline{T}}{\partial N}+k(\underline{T}-T^{ext})=0,
\notag
\end{equation}%
assuming the surface heat impulse is proportional to the temperature
difference with the exterior. Temperature $T^{-}$ before collision is known.

Quantity%
\begin{equation*}
B_{i}^{p}\left[ \beta _{i}\right] +\vec{H}_{i}^{p}\cdot \mathbf{grad}\left[
\beta _{i}\right] \geq 0,
\end{equation*}%
is the dissipated work due to the microscopic motions producing the phase
change. Because the thermal effects are mainly due to the macroscopic
velocity discontinuities, we assume it is negligible compared to the
dissipated work%
\begin{equation*}
\mathcal{T}=\Sigma :D(\frac{\vec{U}^{+}+\vec{U}^{-}}{2})=2k_{v}\left( D(%
\frac{\vec{U}^{+}+\vec{U}^{-}}{2})\right) ^{2},
\end{equation*}%
due to the macroscopic motion. We assume also that in the internal energy
quadratic quantity%
\begin{equation*}
\frac{k}{2}\left[ \left\vert \mathbf{grad}\vec{\beta}\right\vert ^{2}\right]
,
\end{equation*}%
is negigible compared to%
\begin{equation*}
C\left[ T\right] +l_{a}\left[ \beta _{3}\right] .
\end{equation*}%
Thus the thermal equation becomes%
\begin{gather}
C\left[ T\right] +l_{a}\left[ \beta _{3}\right] -\lambda \Delta \underline{T}
\notag \\
=\Sigma :D(\frac{\vec{U}^{+}+\vec{U}^{-}}{2})=2k_{v}\left( D(\frac{\vec{U}%
^{+}+\vec{U}^{-}}{2})\right) ^{2}=\mathcal{T},\ in\;\Omega ,\
\label{thermal}
\end{gather}

\section{Closed Form Example}

We assume the solid is struck by an hammer and that the dissipated work is
known
\begin{equation*}
\mathcal{T}=\Sigma :D(\frac{\vec{U}^{+}+\vec{U}^{-}}{2})=2k_{v}\left( D(%
\frac{\vec{U}^{+}+\vec{U}^{-}}{2})\right) ^{2},
\end{equation*}%
neglecting the dissipated work due to phase changes.

The main assumption is that the volume fractions and temperatures are
homogeneous, i.e., their values do not depends on space variable $x$. The
equations become non linear algebraic equations

\begin{gather}
c\left[ \vec{\beta}\right] +\partial I_{C}(\vec{\beta}^{+})+\left\vert
\begin{array}{c}
-\hat{P} \\
-\hat{P} \\
-\frac{l_{a}}{T_{0}}(T^{+}-T_{0}) -\hat{P}%
\end{array}%
\right\vert =0,  \notag \\
C\left[ T\right] +l_{a}\left[ \beta _{3}\right] =\mathcal{T},  \notag \\
\left[ \beta _{1}+\beta _{2}+\beta _{3}\right] =0.  \label{system}
\end{gather}%
We may prove that system (\ref{system}) has one and only one solution
depending on the quantities before collision.

We investigate the situation where a mixture of the three phases can coexist
after the collision
\begin{equation}
0<\beta _{i}^{+}<1.  \notag
\end{equation}%
The equations are%
\begin{gather*}
c\left[ \vec{\beta}\right] +\partial I_{C}(\vec{\beta}^{+})+\left\vert
\begin{array}{c}
-\hat{P} \\
-\hat{P} \\
-\frac{l_{a}}{T_{0}}(T^{+}-T_{0})-\hat{P}%
\end{array}%
\right\vert =0, \\
C\left[ T\right] +l_{a}\left[ \beta _{3}\right] =\mathcal{T}, \\
\beta _{1}^{+}+\beta _{2}^{+}+\beta _{3}^{+}=1,
\end{gather*}%
with
\begin{equation}
\partial I_{C}(\vec{\beta}^{+})=\left[
\begin{array}{c}
D \\
D \\
D%
\end{array}%
\right] ,\ with\ D\in \mathbb{R}^{+}.  \notag
\end{equation}%
We choose initial state
\begin{equation}
\beta _{1}^{-}=\beta _{2}^{-}=\frac{1}{2},~\beta _{3}^{-}=0,  \notag
\end{equation}%
which is an equilibrium when
\begin{equation*}
T^{-}\leq T_{0},
\end{equation*}%
that we assume.

We get from the first equation, subtracting line 1 from line 2 and
line 1 from line 3,
\begin{equation}
\beta _{1}^{+}=\beta _{2}^{+},~c(\beta _{3}^{+}-(\beta _{1}^{+}-\frac{1}{2}%
))=\frac{l_{a}}{T_{0}}(T^{+}-T_{0}).  \notag
\end{equation}
With the last equation, we have
\begin{equation}
c(\beta _{3}^{+}-(\frac{1}{2}-\frac{\beta _{3}^{+}}{2})+\frac{1}{2})=\frac{%
l_{a}}{T_{0}}(T^{+}-T_{0}).  \notag
\end{equation}
It results
\begin{equation}
\beta _{3}^{+}=\frac{2l_{a}}{3cT_{0}}(T^{+}-T_{0}).  \notag
\end{equation}
Note that if there is not dissipation, the temperature has to be equal to
the phase change temperature. From the second equation, we get
\begin{equation}
C(T^{+}-T^{-})+\frac{2l_{a}^{2}}{3cT_{0}}(T^{+}-T_{0})=\mathcal{T},  \notag
\end{equation}
giving
\begin{equation}
(C+\frac{2l_{a}^{2}}{3cT_{0}})T^{+}=\mathcal{T}+CT^{-}+\frac{2l_{a}^{2}}{3c}.
\notag
\end{equation}

Function $\mathcal{T\rightarrow }T^{+}$ is increasing. This is the solution
as long as
\begin{equation}
0\leq \beta _{3}^{+}\leq 1,  \notag
\end{equation}%
or
\begin{equation}
0\leq T^{+}-T_{0}\leq \frac{2l_{a}}{3cT_{0}}.  \notag
\end{equation}%
To satisfy these conditions, the dissipated work $\mathcal{T}$ has to verify
\begin{equation}
0\leq \mathcal{T}+C(T^{-}-T_{0})\leq \frac{3cCT_{0}+2l_{a}^{2}}{2l_{a}}.
\notag
\end{equation}%
In case there is dissipation, $c>0$, the three phases may coexist at
temperatures different from $T_{0}$ whereas the temperature has to be equal
to $T_{0}$ in case there is not dissipation, $c=0$. Of course, depending on
the temperature before collision, the dissipated work $\mathcal{T}$ has to
be not too small and not too large. The complete phase change occurs for a
dissipated work large enough. For a very weak hammer stroke there is only an
increase of temperature and no phase change.

\section{ The PDE System: Existence and Uniqueness of Solutions}

In this section we introduce some convex sets in order to eliminate one of
volume fractions of $\vec{\beta}^{+}$ and reformulate equation of motion for
the microscopic motions \eqref{equU}--\eqref{convex} with \eqref{thermal}
and with the associated boundary conditions. Then an existence and
uniqueness of solutions of the problem theorem is proved.

\subsection{The Problem}

The state
\begin{equation*}
\vec{\beta}^{-},\ T^{-},
\end{equation*}%
with $\func{grad}T^{-}=0$ and velocity%
\begin{equation*}
\vec{U}^{-}=0,
\end{equation*}%
before collision are given. The unknowns are%
\begin{equation*}
\vec{\beta}^{+},\ T^{+},\ \vec{U}^{+},
\end{equation*}%
the state and velocity after collision. The equation are

\begin{itemize}
\item the mass balance%
\begin{equation*}
(\beta _{1}+\beta _{2}+\beta _{3})^{+}=(\beta _{1}+\beta _{2}+\beta
_{3})^{-}=1,
\end{equation*}

\item the equation of motion for the macroscopic motion
\begin{gather*}
\rho \vec{U}^{+}-k_{v}\func{div}D(\vec{U}^{+})=0,\ in\;\Omega , \\
k_{v}\frac{\partial \vec{U}^{+}}{\partial N}=\vec{G}^{p},\ on\ \Gamma _{1},%
\vec{U}^{+}=\vec{U}^{-}=0,\ on\ \Gamma _{0},
\end{gather*}

\item the equation of motion for the microscopic motions%
\begin{gather*}
c\left[ \vec{\beta}\right] -\upsilon \Delta \left[ \vec{\beta}\right]
-k\Delta \vec{\beta}^{+}+\vec{B}_{reca}^{fe} \\
+\left\vert
\begin{array}{c}
-P-CT^{+}\log T^{+} \\
-P-CT^{+}\log T^{+} \\
-\frac{l_{a}}{T_{0}}(T^{+}-T_{0})-P-CT^{+}\log T^{+}%
\end{array}%
\right\vert =0,\ in\;\Omega , \\
\vec{B}_{reca}^{fe}\in \partial I_{C}(\vec{\beta}^{+}), \\
\upsilon \frac{\partial \left[ \vec{\beta}\right] }{\partial N}+k\frac{%
\partial \vec{\beta}^{+}}{\partial N}=0,\ on\ \partial \Omega \text{,}
\end{gather*}

\item the energy balance

\begin{gather*}
C\left[ T\right] +l_{a}\left[ \beta _{3}\right] -\lambda \Delta \underline{T}%
=2k_{v}\left( D(\frac{\vec{U}^{+}+\vec{U}^{-}}{2})\right) ^{2},\ in\;\Omega ,
\\
\frac{\partial \underline{T}}{\partial N}=0,\ on\ \partial \Omega .
\end{gather*}
\end{itemize}

\subsection{Some Convex Sets}

Let define convex set
\begin{equation}
\tilde {C} =\{(\gamma_{1},\gamma_{2},\gamma_{3})\in \mathbb{R}^{3};0\leq
\gamma_{i}\leq 1;\gamma_{1}+\gamma_{2}+\gamma_{3} = 1\}\,.  \label{defCtilde}
\end{equation}%
It is a plane triangle in $\mathbb{R}^3$, intersection of tetrahedron $C$
and plane $\gamma_{1}+\gamma_{2}+\gamma_{3} = 1$.

Let define function
\begin{equation}
\hat{I}_{0}(\vec{\gamma)}=I_{0}(\gamma_{1}+\gamma_{2}+\gamma_{3} - 1)
\end{equation}

We have
\begin{equation}
I_{\tilde {C}}=\hat{I}_{0}+I_{{C}},
\end{equation}

and

\begin{equation}
\partial I_{\tilde {C}} \supset \partial \hat{I}_{0}+ \partial I_{{C}},
\end{equation}

We have equality
\begin{equation}
\partial I_{\tilde{C}}=\partial \hat{I}_{0}+\partial I_{{C}},  \label{propC}
\end{equation}%
because the interior of convex set $C$ is not empty, \cite{moreau 1966}.

The plane triangle in $\mathbb{R}^{2}$
\begin{equation*}
K:=\left\{
\begin{array}{l|l}
\lbrack \gamma _{1},\gamma _{2}]\in \mathbb{R}^{2} & 0\leq \gamma
_{1},\,\gamma _{2}\leq 1,\ \gamma _{1}+\gamma _{2}\leq 1%
\end{array}%
\right\} .
\end{equation*}%
is a convex set. It is easy to prove%
\begin{equation}
\left\vert
\begin{array}{c}
D_{1} \\
D_{2} \\
D_{3}%
\end{array}%
\right\vert \in \partial I_{\tilde{C}}(\vec{\beta})\iff \left\vert
\begin{array}{c}
D_{2}-D_{1} \\
D_{3}-D_{1}%
\end{array}%
\right\vert \in \partial I_{K}(\beta _{2}^{+};\beta _{3}^{+}).  \label{propK}
\end{equation}

\subsection{The Equation of Motion for the Microscopic Motions}

By using relationship (\ref{propC}), the equations (\ref{equU})-(\ref{equaP}%
) become%
\begin{gather}
c\left[ \vec{\beta}\right] -\upsilon \Delta \left[ \vec{\beta}\right]
-k\Delta \vec{\beta}^{+}+\vec{B}_{reca}^{fe}+\left\vert
\begin{array}{c}
0 \\
0 \\
-\frac{l_{a}}{T_{0}}(T^{+}-T_{0})%
\end{array}%
\right\vert =0,\ in\;\Omega ,  \label{solv2bis} \\
\vec{B}_{reca}^{fe}\in \partial I_{\tilde{C}}(\vec{\beta}^{+}),  \notag \\
\upsilon \frac{\partial \left[ \vec{\beta}\right] }{\partial N}+k\frac{%
\partial \vec{\beta}^{+}}{\partial N}=0,\ on\ \partial \Omega \text{,}
\notag
\end{gather}

An easy computation using relationship (\ref{propK}) shows that system (\ref%
{solv2bis}) is equivalent to%
\begin{gather}
\left(
\begin{array}{cc}
2 & 1 \\
1 & 2%
\end{array}%
\right) \left\vert
\begin{array}{c}
c\left[ \beta _{2}\right] -\upsilon \Delta \left[ \beta _{2}\right] -k\Delta
\beta _{2}^{+} \\
c\left[ \beta _{3}\right] -\upsilon \Delta \left[ \beta _{3}%
\right] -k\Delta \beta _{3}^{+}%
\end{array}%
\right\vert +\left\vert
\begin{array}{c}
\tilde{D}_{2}^{reac} \\
\tilde{D}_{3}^{reac}%
\end{array}%
\right\vert +\left\vert
\begin{array}{c}
0 \\
-\frac{l_{a}}{T_{0}}(T^{+}-T_{0})%
\end{array}%
\right\vert =0,\ in\;\Omega ,  \notag \\
\left\vert
\begin{array}{c}
\tilde{D}_{2}^{reac} \\
\tilde{D}_{3}^{reac}%
\end{array}%
\right\vert \in \partial I_{K}(\beta _{2}^{+};\beta _{3}^{+}),  \notag \\
\upsilon \frac{\partial \left[ \beta _{2}\right] }{\partial N}+k\frac{%
\partial \beta _{2}^{+}}{\partial N}=0,\ \upsilon \frac{\partial \left[
\beta _{3}\right] }{\partial N}+k\frac{\partial \beta _{3}^{+}}{\partial N}%
=0,\ on\ \partial \Omega \text{,}  \label{nonclassic}
\end{gather}

\subsection{Another Equation of Motion for the Microscopic
Motions}

By choosing the more sophisticated interfacial energy (\ref{dif1}) and
pseudo-potential of dissipation (\ref{dif2}) with%
\begin{equation*}
k_{1}=0,~k_{2}=k_{3}=k,\ c_{1}=0,~c_{2}=c_{3}=c,\ \upsilon _{1}=0,~\upsilon
_{2}=\upsilon _{3}=\upsilon .
\end{equation*}%
Let us note that phase one still intervenes in the physical properties of
the alloy because there are no voids, see \cite{fremya1996}, \cite%
{fremond2011}
\begin{equation*}
\beta _{1}^{-}+\beta _{2}^{-}+\beta _{3}^{-}=\beta _{1}^{+}+\beta
_{2}^{+}+\beta _{3}^{+}=1.
\end{equation*}%
With these choices the equations giving the $\beta _{i}^{+}$ are%
\begin{gather*}
\left\vert
\begin{array}{c}
0 \\
c\left[ \beta _{2}\right] -\upsilon \Delta \left[ \beta _{2}\right] -k\Delta
\beta _{2}^{+} \\
c\left[ \beta _{3}\right] -\upsilon \Delta \left[ \beta _{3}%
\right] -k\Delta \beta _{3}^{+}%
\end{array}%
\right\vert +\left\vert
\begin{array}{c}
D_{1}^{reac} \\
D_{2}^{reac} \\
D_{3}^{reac}%
\end{array}%
\right\vert +\left\vert
\begin{array}{c}
0 \\
0 \\
-\frac{l_{a}}{T_{0}}(T^{+}-T_{0})%
\end{array}%
\right\vert =0,\ in\;\Omega , \\
\left\vert
\begin{array}{c}
D_{1}^{reac} \\
D_{2}^{reac} \\
D_{3}^{reac}%
\end{array}%
\right\vert \in \partial I_{\tilde{C}}(\vec{\beta}^{+}), \\
\upsilon \frac{\partial \left[ \beta _{2}\right] }{\partial N}+k\frac{%
\partial \beta _{2}^{+}}{\partial N}=0,\ \upsilon \frac{\partial \left[
\beta _{3}\right] }{\partial N}+k\frac{\partial \beta _{3}^{+}}{\partial N}%
=0,\ on\ \partial \Omega \text{.}
\end{gather*}%
By using relationship (\ref{propK}) we get%
\begin{gather}
\left\vert
\begin{array}{c}
c\left[ \beta _{2}\right] -\upsilon \Delta \left[ \beta _{2}\right] -k\Delta
\beta _{2}^{+} \\
c\left[ \beta _{3}\right] -\upsilon \Delta \left[ \beta _{3}%
\right] -k\Delta \beta _{3}^{+}%
\end{array}%
\right\vert +\left\vert
\begin{array}{c}
D_{2}^{reac} \\
D_{3}^{reac}%
\end{array}%
\right\vert +\left\vert
\begin{array}{c}
0 \\
-\frac{l_{a}}{T_{0}}(T^{+}-T_{0})%
\end{array}%
\right\vert =0,\ in\;\Omega ,  \notag \\
\left\vert
\begin{array}{c}
D_{2}^{reac} \\
D_{3}^{reac}%
\end{array}%
\right\vert \in \partial I_{K}(\beta _{2}^{+};\beta _{3}^{+}),  \notag \\
\upsilon \frac{\partial \left[ \beta _{2}\right] }{\partial N}+k\frac{%
\partial \beta _{2}^{+}}{\partial N}=0,\ \upsilon \frac{\partial \left[
\beta _{3}\right] }{\partial N}+k\frac{\partial \beta _{3}^{+}}{\partial N}%
=0,\ on\ \partial \Omega \text{,}  \label{classic}
\end{gather}

\subsection{The PDE\ System}

We introduce notations to make precise the mathematical formulation of the
problem and prove the existence and uniqueness theorem.

\subsubsection{Notation}

In order to give a precise formulation of our problem, let us denote by $%
\Omega $ a bounded, convex set in $\mathbb{R}^{3}$ with $C^{1}$ boundary $%
\Gamma $. Let $(\Gamma _{0},\Gamma _{1})$ be a partition of $\partial \Omega
$ into two measurable sets such that both $\Gamma _{0}$ has positive surface
measure and it's Lipschitz. Finally, we introduce the Hilbert triplet $%
(V,H,V^{\prime })$ where
\begin{equation}
H:=L^{2}(\Omega )\quad \hbox{and}\quad V:=W^{1,2}(\Omega )  \label{spazi}
\end{equation}%
and identify, as usual, $H$ (which stands either for the space $L^{2}(\Omega
)$ or for $(L^{2}(\Omega ))^{3}$) with its dual space $H^{\prime }$, so that
$V\hookrightarrow H\hookrightarrow V^{\prime }$ with dense and continuous
embeddings. Moreover, we denote by $\Vert \cdot \Vert _{X}$ the norm in some
space $X$ and by $\duav{\cdot,\cdot}$ the duality pairing between $V$ and $%
V^{\prime }$ and by $(\cdot ,\cdot )$ the scalar product in $H$. In the
space $V$ we introduce the inner product
\begin{equation*}
(u,v)_{V}:=(\grad u,\grad v)+\int_{\Gamma }u_{|_{\Gamma }}v_{|_{\Gamma }}%
\mbox{
\ for any $ u, v \in V $.}
\end{equation*}%
We define the Hilbert space
\begin{equation}
\mathbf{W}:=\{\mathbf{v}\in (V)^{3}:\,\mathbf{v}_{|_{\Gamma _{0}}}=\mathbf{0}%
\}  \label{defiW}
\end{equation}%
endowed with the usual norm. In addition, we introduce on $\mathbf{W}\times
\mathbf{W}$ a bilinear symmetric continuous form $a(\cdot ,\cdot )$ defined
by
\begin{equation}\notag
a(\ub,{\bf v}):=\int_\Omega \e(\ub):\e(\vb):= \sum_{i,j=1}^3\int_\Omega\e_{ij}(\ub)\e_{ij}({\bf v}).
\end{equation}
Note here that (since $\Gamma _{0}$ has positive measure), thanks to Korn's
inequality (cf., e.g., \cite{ciarlet}, \cite[p.~110]{DL}), there exists a
positive constant $c$ such that
\begin{equation}
a(\mathbf{v},\mathbf{v})\geq c\Vert \mathbf{v}\Vert _{\mathbf{W}}^{2}\quad
\forall \mathbf{v}\in \mathbf{W}.  \label{aequiv}
\end{equation}%
Moreover, we introduce the space
\begin{equation*}
W_{\Gamma _{0}}^{1,q}(\Omega ):=\{v\in W^{1,q}(\Omega )\,:\,v=0\quad %
\hbox{on }\Gamma _{0}\}
\end{equation*}%
and we do not distinguish in the notation the spaces $W^{1,p}(\Omega )$ and $%
(W^{1,p}(\Omega ))^{3}$ as well as the norms $\Vert \cdot \Vert _{1,p}$ and $%
\Vert \cdot \Vert _{p}$ which stand for the usual norms in $W^{1,p}(\Omega )$%
, $W^{1,p}(\Omega )^{3}$ and $L^{p}(\Omega $, $(L^{p}(\Omega ))^{3}$,
respectively.

We use equation (\ref{classic}) as equation of motion for the microscopic
motions. Then, choosing for simplicity and without any loss of generality %
$\rho =k_{v}=c=\upsilon +k=\lambda =1$, $l_{a}=T_{0}$, and $%
\beta _{2}^{-}=\beta _{3}^{-}=0$, $\vec{U}^{-}=0$, we can rewrite our
system, coupling \eqref{equU}, \eqref{convex}, \eqref{equaP}%
, \eqref{thermal}, and \eqref{classic}, in the new variables $(\vartheta ,%
\ub,\chi _{2},\chi _{3})$, corresponding to the previous $(T^{+},\vec{U}%
^{+},\beta _{2},\beta _{3})$, as
\begin{align}
& \vartheta +T_{0}\chi _{3}-\Delta \vartheta =f+\e(\ub):\e(\ub),\
\mbox{ a.e.~in
$\Omega $,}  \label{eteta} \\
& \left( \hspace{-0.15cm}%
\begin{array}{c}
\chi _{2} \\
\chi _{3}%
\end{array}%
\hspace{-0.15cm}\right) -\left( \hspace{-0.15cm}%
\begin{array}{c}
\Delta \chi _{2} \\
\Delta \chi _{3}%
\end{array}%
\hspace{-0.15cm}\right) +\partial I_{K}(\chi _{2},\chi _{3})\ni \left(
\hspace{-0.15cm}%
\begin{array}{c}
0 \\
\vartheta -T_{0}%
\end{array}%
\hspace{-0.15cm}\right) \!,\ \mbox{ a.e.~in $ \Omega $,}  \label{echi} \\
& \ub-\dive(\e(\ub))=0,\ \mbox{ a.e.~in $\Omega$,}  \label{equ} \\
& \dn\vartheta =\Pi ,\quad \mbox{a.e.~on }\Gamma , \\
& \dn\chi _{i}=h_{i},\ i=2,3,\quad \mbox{a.e.~on }\Gamma , \\
& \ub=0,\quad \mbox{a.e.~on }\Gamma _{0},  \label{bouu1} \\
& \e(\ub)\cdot \mathbf{n}=\mathbf{g},\ \mbox{a.e.~on }\Gamma _{1}.
\label{bouu2}
\end{align}%
Notice that $\e_{ij}(\ub):=(\partial _{x_{j}}u_{i}+\partial
_{x_{i}}u_{j})/2, $ $i,j=1,2,3,$ are the components of the standard
linearized \textsl{strain tensor} $\e(\ub)$ and $\mathbf{n}$ stands for the
outward normal vector to $\Gamma .$ Concerning data, $f:\Omega \rightarrow
\mathbb{R}$ represents a known source term (it was $f=-\Delta
T^{-}+T^{-}+T_{0}\beta _{3}^{-}$ in the previous sections), $h_{i}:\Omega
\rightarrow \mathbb{R}$, $i=2,3$ are given boundary data (it was %
$h_{i}=\upsilon \dn\beta _{i}^{-}$ in previous sections and so
we should have $h_{i}=0$ here, but we prefer to let it be more general,
different from 0), $\Pi :\Gamma \rightarrow \mathbb{R}$ denotes an energy
flux coming from the exterior of the system (it was $\Pi =-\dn T^{-}$ in the
previous sections), and $\mathbf{g}:\Gamma _{1}\rightarrow \mathbb{R}^{3}$
yields the external contact force applied to $\Gamma _{1}$ (it was $\mathbf{g%
}=\vec{G}^{p}$ in the previous sections). Moreover the maximal monotone
graph $\partial I_{K},$ representing the subdifferential of the indicator
function $I_{K}$ of the plane triangle $K$. Set $K$ is convex and contains
the admissible phase proportions. We also notice that $I_{K}(\chi _{2},\chi
_{3})=0$ if $[\chi _{2},\chi _{3}]\in K,$ \ $=+\infty $ otherwise. For
definitions and basic properties of maximal monotone operators and
subdifferentials of convex functions, we refer, for instance, to \cite{Br},
\cite{moreau 1966}.

Let us comment now on the fact that, with the help of the usually considered
boundary conditions (\ref{bouu1})--(\ref{bouu2}) (see, e.g., \cite{CFV}),
from \eqref{equ} it turns out that (cf.~ \cite[Thm.~6.2, p.~168]{DL} and
\cite[Thm.~1.1, p.~437]{SW94}), at almost any time $t,$ the velocity
$\ub$ can be completely determined in terms of the datum $\mathbf{g}$. Thus, we
may introduce the operator $F_{\mathbf{g}}$, which maps $\mathbf{g}$ into $\e%
(\ub):\e(\ub)$ where $\mathbf{u}$ stands for the related solution of %
\eqref{equ}, (\ref{bouu1})--(\ref{bouu2}). Then, you consider the following
system, denoted by $(SMA)$, whose unknowns are now the absolute temperature $%
\teta$ and the phase variables~$\chi_2,\chi_3$.
\begin{align}
&\vartheta +T_0\chi_3 -\Delta\vartheta = f+F_{\mathbf{g}}, \
\mbox{ a.e.~in
$\Omega $,}  \tag{$SMA_1$}  \label{equteta} \\
&\left( \hspace{-0.15cm}
\begin{array}{c}
\chi_2 \\
\chi_3%
\end{array}
\hspace{-0.15cm} \right)-\left( \hspace{-0.15cm}
\begin{array}{c}
\Delta\chi_1 \\
\Delta\chi_2%
\end{array}
\hspace{-0.15cm} \right) +\left( \hspace{-0.15cm}
\begin{array}{c}
\xi_2 \\
\xi_3%
\end{array}
\hspace{-0.15cm} \right) = \left( \hspace{-0.15cm}
\begin{array}{c}
0 \\
\vartheta-T_0%
\end{array}
\hspace{-0.15cm} \right)\!, \ \mbox{ a.e.~in $ \Omega $,}  \tag{$SMA_2$}
\label{equchi} \\
&\left( \hspace{-0.15cm}
\begin{array}{c}
\xi_2 \\
\xi_3%
\end{array}
\hspace{-0.15cm} \right)\in \partial I_K(\chi_2, \chi_3)\
\mbox{ a.e.~in $
\Omega $,}  \tag{$SMA_3$}  \label{indica} \\
&\dn\vartheta=\Pi,\quad\mbox{a.e.~on }\Gamma,  \tag{$SMA_4$}  \label{bouteta}
\\
&\dn\chi_i = h_i, \ i = 2, 3, \quad\mbox{a.e.~on }\Gamma.  \tag{$SMA_5$}
\label{bouchi}
\end{align}

\begin{remark}
\label{exte} Let us note here that in this formulation of
our system we consider just the case $\beta _{2}^{-}=\beta _{3}^{-}=0$, $%
\vec{U}^{-}=0$ for simplicity, however our existence and uniqueness results
could be extended to the case these quantities are different from zero
sufficiently regular. Moreover, we could handle the case where inclusion %
\eqref{classic} is replaced by \eqref{nonclassic} in the same way due to the
fact that the only difference would be having a positive definite matrix
acting on $(\chi _{2},\chi _{3})$ instead of the identity matrix, that could
be treated in a very similar manner. 

In case {$\vec{U}^{-}\neq 0$ with }$D(${$\vec{U}^{-})\neq 0$, we choose a
pseudo-potential of dissipation satisfying}\begin{eqnarray*}
&&\left( 0,(B_{i}^{p}-B_{i}^{fe}),(\vec{H}_{i}^{p}-\vec{H}_{i}^{fe}),-2\vec{Q}_{p}\right)  \\
&\in &\partial \Phi (D(\frac{\vec{U}^{-}}{2}),\left[ \beta _{i}\right] ,\mathbf{grad}\left[ \beta _{i}\right] ,\mathbf{grad}\underline{T},\underline{T}),
\end{eqnarray*}implying that there is no collision if there is no hammer stroke, $\vec{G}^{p}=0$. Nevertheless the schematic simple pseudo-potential of dissipation (\ref{pseudo}) is adapted to account for what occurs for large hammer
strokes, i.e., for large $D(\vec{U}^{+}+\vec{U}^{-})/2$.
\end{remark}

\subsection{The Existence and Uniqueness result}

Let us first make precise the statement regarding the operator $F_{\mathbf{g}%
}$, whose proof can be found as result of \cite[Thm.~6.2, p.~168]{DL}, a
slight modification of \cite[Thm.~1.1, p.~437]{SW94}, and \cite[Thm.\
6.3-.6, p.\ 296]{ciarlet} (cf.\ also \cite[p.\ 260]{necas}).

\begin{lemma}
\label{ureg} (Higher integrability of the gradient) Given $\mathbf{g}\in
L^2(\Gamma_1)^3)$, there exists a unique $\ub\in \mathbf{W}$ \emph{weak
solution} of
\begin{equation}  \label{equweak}
a(\ub, \vb)=\int_{\Gamma_1}\mathbf{g}\cdot \vb_{|\Gamma_1}\quad \forall \vb%
\in \mathbf{W}.
\end{equation}
Moreover, there exists $\gamma>0$ such that the following statement holds:
whenever $2<p<2+\gamma$, $\mathbf{g}\in (L^{2p/3}(\Gamma_1))^3,$, and $\ub$
is a \emph{weak solution} of \eqref{equweak}, then $\ub\in
(W^{1,p}(\Omega))^3$ and
\begin{equation*}
\|\ub\|_{1,p}\leq C\|\mathbf{g}\|_{L^{2p/3}(\Gamma_1)},
\end{equation*}
where $\gamma$ depends only on $\Omega$, $\Gamma_0$, and $C$ depends only on
$p$ and these quantities. Hence, if we denote by $F_{\mathbf{g}}$ the
operator
\begin{equation*}
F_{\mathbf{g}}\,:\, \mathbf{g}\in (L^{2p/3}(\Gamma_1))^3\mapsto \e(\ub):\e(%
\ub) \in L^{p/2}(\Omega),
\end{equation*}
where $\ub$ denotes the unique \emph{weak solution} of \eqref{equweak}
corresponding to $\mathbf{g}$, then, there holds
\begin{equation}  \label{estiF}
\|F_{\mathbf{g}}\|_{p/2}\leq C\|\mathbf{g}\|_{(L^{2p/3}(\Gamma_1))^3},
\end{equation}
where $C$ deoends on $\Omega$, $\Gamma_0$, and $p$. Finally, if $\Omega$ is
a $C^2$ domain, the closures of $\Gamma_0$ and $\Gamma_1$ do not intersect
and $\mathbf{g}\in (L^2(\Gamma_1))^3$, then $\ub\in (H^2(\Omega))^3$ and the
following two estimates hold true
\begin{align}  \label{cdu}
&\| \ub \|_{(H^2(\Omega))^3} \leq C_1\|\mathbf{g}\|_{(L^2(\Gamma_1))^3}\,, \\
&\|F_{\mathbf{g}_1}-F_{\mathbf{g}_2}\|_{4/3}\leq C_2\|\mathbf{g}_1-\mathbf{%
g_2}\|_{(L^2(\Gamma_1))^3},\,
\end{align}
for every $\mathbf{g}_i\in (L^2(\Gamma_1))^3$, $i=1,2$, where the constant $%
C_2$ depends on the problem data and also on the $\| \ub \|_{(H^2(%
\Omega))^3} $.
\end{lemma}

\begin{remark}
\label{ugen} Let us note that we could actually generalize here the form of
the elasticity bilinear form $a$ including an elasticity matrix which do not
need to be exactly the identity matrix but it needs to be a symmetric
matrix, positive definite with $L^\infty$ coefficients. Moreover we could
include a non-zero volume percussion on the right hand side in \eqref{equweak}.
The regularity requirement we would need in order to apply the results of
\cite{SW94} is $L^p(\Omega)^3$. In fact, also the third part of the Lemma
extends to the case of an anisotropic and inhomogeneous material, for which
the elasticity tensors $\mathrm{R}_e$ is of the form $\mathrm{R}%
_e=(g_{ijkh}) $, with functions
\begin{equation}  \label{funz_g-l}
g_{ijkh} \in \mathrm{C}^{1}(\Omega)\,, \quad i,j,k,h=1,2,3,
\end{equation}
satisfying the classical symmetry and ellipticity conditions (with the usual
summation convention)
\begin{equation}  \label{ellipticity}
\begin{array}{ll}
& \!\!\!\!\!\!\!\! \!\!\!\!\!\!\!\! g_{ijkh}=g_{jikh}=g_{khij}, \quad
i,j,k,h=1,2,3 \\
& \!\!\!\!\!\!\!\! \!\!\!\!\!\!\!\! \exists \, C_1>0 \,: \, g_{ijkh}
\xi_{ij}\xi_{kh}\geq C_1\xi_{ij}\xi_{ij}, \, \text{for all } \xi_{ij}\colon
\xi_{ij}= \xi_{ji}\,,\, i,j=1,2,3\,.%
\end{array}%
\end{equation}
Notice moreover that in the latter part of Lemma~\ref{ureg} the
compatibility condition between $\Gamma_0$ and $\Gamma_1$ is necessary.
Indeed, without the latter geometric condition, the elliptic regularity
results ensuring the (crucial) $(H^2(\Omega))^3$-regularity of $\ub$ may
fail to hold, see~\cite[Chap.~VI,~Sec.~6.3]{ciarlet}. In particular, the
last estimate \eqref{cdu}, which will be used in order to prove the
stability estimate \eqref{stability}, is obtained by means of Young
inequality and the standard stability estimate for $\ub$
\begin{equation*}
\|\ub_1-\ub_2\|_{\mathbf{W}}\leq C\|\mathbf{g}_1-\mathbf{g_2}%
\|_{(L^2(\Gamma_1))^3},
\end{equation*}
which can be proved by testing the diferences of \eqref{equweak}'s by $\ub_1-%
\ub_2$.
\end{remark}

Now we are in the position to state and prove our existence and uniqueness
result

\begin{theorem}
\label{main} (Existence and Uniqueness) Assume that $T_0\in \mathbb{R}$,
there exists $\gamma>0$ such that $\mathbf{g}\in(L^{2p/3}(\Gamma_1))^3$ for
some $2<p<2+\gamma$, $\Pi\in W^{1-1/p,p}(\Gamma)$, $h_i\in W^{1/2,
2}(\Gamma) $, $i=2,3$, $f\in L^2(\Omega)$. Then there exists a unique
solution $(\teta, \chi_2, \chi_3)\in (H^1(\Omega)\cap
W^{2,p/2}(\Omega))\times H^2(\Omega)\times H^2(\Omega)$ of system (SMA) such
such that 
\begin{equation*}
\|\teta\|_{W^{2,p/2}(\Omega)}+\sum_{i=2}^3\|\chi_i\|_{H^2(\Omega)}\leq C,
\end{equation*}
where $C$ depends only on the problem data and on $p$. Finally, if $\Omega$
is a $C^2$ domain, the closures of $\Gamma_0$ and $\Gamma_1$ do not
intersect and $\mathbf{g}_i\in (L^2(\Gamma_1))^3$, $f_i\in H$, $(h_i)_1, \,
(h_i)_2\in L^2(\Gamma)$, $i=1,2$, then the following stability estimate
holds true 
\begin{equation}  \label{stability}
\|\teta_1-\teta_2\|_V+\sum_{i=2}^3\|(\chi_i)_1-(\chi_i)_2\|_{V}\leq C\left(
\|\mathbf{g}_1-\mathbf{g}_2\|_{(L^2(\Gamma_1))^3}+
\sum_{i=2}^3\|(h_i)_1-(h_i)_2\|_{L^2(\Gamma)}+\|f_1-f_2\|_H\right),
\end{equation}
where $(\teta_i, (\chi_2)_i, (\chi_3)_i)$ are two solutions corresponding to
data $\mathbf{g}_i\in (L^2(\Gamma_1))^3$, $f_i\in H$, $(h_i)_1, \,
(h_i)_2\in L^2(\Gamma)$, $i=1,2$, respectively, and the constant $C$ depends
only on the data of the problem.
\end{theorem}

\paragraph{Proof of Theorem~\protect\ref{main}.}

The main idea of the proof of existence of solutions is to use a fixed point
argument (Shauder theorem), that is, to prove that a suitable operator
\begin{equation*}
D\,:\,L^{2}(\Omega )\times L^{2}(\Omega )\rightarrow H^{2}(\Omega )\times
H^{2}(\Omega )\hookrightarrow \hookrightarrow L^{2}(\Omega )\times
L^{2}(\Omega )
\end{equation*}%
admits at least a fixed point. To do that we employ a standard Shauder
theorem by proving that the operator $D$ is continuous and compact. First we
fix $(\tilde{\chi}_{2},\tilde{\chi}_{3})\in (L^{2}(\Omega )\times
L^{2}(\Omega ))\cap K$ and we substitute $\tilde{\chi}_{3}$ in  %
\eqref{equteta} in place of $\chi _{3}$ and we find (cf. \cite[Ch.~2]{gris}%
) a unique $\tilde{\teta}\in V\cap W^{2,p/2}(\Omega )$, $\tilde{\teta}%
=D_{1}((\tilde{\chi}_{2},\tilde{\chi}_{3}))$ solution of %
\eqref{equteta} coupled with \eqref{bouteta} and since $|%
\tilde{\chi}_{3}|\leq C$, then $\Vert \tilde{\teta}\Vert _{V\cap
W^{2,p/2}(\Omega )}\leq S_{1}$ independently of $\tilde{\chi}_{3}$. Consider
now the differential inclusion \eqref{equchi}, %
\eqref{indica} with $\tilde{\teta}$ in place of $\teta$. Then, by applying
known regularity results (cf., e.g., \cite{BCP}), we have that there exists
a unique solution $(\chi _{2},\chi _{3})=D_{2}(\tilde{\teta})\in
H^{2}(\Omega )\times H^{2}(\Omega )$ and 
\begin{equation*}
\begin{array}{r}
\displaystyle\sum_{i=2}^{3}\left\{ \Vert \chi _{i}\Vert _{H^{2}(\Omega
)}^{2}+\Vert \xi _{i}\Vert _{H}^{2}\right\} \leq S_{2}(1+\Vert \vartheta
\Vert _{H}^{2})\leq S_{3},%
\end{array}%
\end{equation*}
where $S_{2}$ and $S_{3}$ depend only on the data of the problem, but not on
the choice of $(\tilde{\chi}_{2},\tilde{\chi}_{3})$. From the above argument
it follows that the operator $D::\,L^{2}(\Omega )\times L^{2}(\Omega
)\rightarrow H^{2}(\Omega )\times H^{2}(\Omega )$ specified by $D=D_{2}\circ
D_{1}$ is a compact operator in $H\times H$. It remains to prove that the
operator is continuous, but this follows from the following estimate. Let $%
\teta_{i}=D_{1}((\tilde{\chi}_{2})_{i}),(\tilde{\chi}_{3})_{i})$, $i=1,2$,
then, we easily get
\begin{equation*}
\Vert \teta_{1}-\teta_{2}\Vert _{V}^{2}\leq C\Vert (\tilde{\chi}_{3})_{1}-(%
\tilde{\chi}_{3})_{2}\Vert _{H}^{2}.
\end{equation*}%
Moreover, by monotonicity arguments, letting $((\chi _{2})_{i},(\chi
_{3})_{i})=D_{2}(\teta_{i})$, $i=1,2$, we get 
\begin{equation*}
\sum_{i=2}^{3}\Vert (\chi _{i})_{1}-(\chi _{i})_{2}\Vert _{V}^{2}\leq C\Vert %
\teta_{1}-\teta_{2}\Vert _{H}^{2}\leq C\Vert (\tilde{\chi}_{3})_{1}-(\tilde{%
\chi}_{3})_{2}\Vert _{H}^{2}
\end{equation*}
which implies the continuity of the operator $D$. Uniqueness of solutions
follows by taking the differences of two equations %
\eqref{equteta} and \eqref{equchi} and testing the first by
the differences of $\teta$'s and the second by differences of $(\chi
_{2},\chi _{3})$'s, solutions associated to the same data. This, thanks to a
cancellation and to monotonicity arguments leads to the estimate 
\begin{equation*}
\sum_{i=2}^{3}\Vert (\chi _{i})_{1}-(\chi _{i})_{2}\Vert _{V}^{2}+\Vert \teta%
_{1}-\teta_{2}\Vert _{V}^{2}\leq 0.
\end{equation*}
This concludes the proof of the first part of the Theorem. The stability
estimate \eqref{stability} follows from the following estimate. We take the
differences of the two equations \eqref{equteta} and %
\eqref{equchi} corresponding to the two solutions $(\teta%
_{i},(\chi _{2})_{i},(\chi _{3})_{i})$, $i=1,2$, and test them by $\teta_{1}-%
\teta_{2}$ and $((\chi _{2})_{1}-(\chi _{2})_{2},(\chi _{3})_{1}-(\chi
_{3})_{2})$, respectively. Then summing up the two resulting equations, we
get 
\begin{align}
\Vert \teta_{1}-\teta_{2}\Vert _{V}^{2}+\sum_{i=2}^{3}\Vert (\chi
_{i})_{1}-(\chi _{i})_{2}\Vert _{V}^{2}& \leq C\left( \sum_{i=2}^{3}\Vert
(h_{i})_{1}-(h_{i})_{2}\Vert _{L^{2}(\Gamma )}^{2}+\Vert f_{1}-f_{2}\Vert
_{H}^{2}\right)   \notag \\
& \quad +\Vert F_{\mathbf{g}_{1}}-F_{\mathbf{g}_{2}}\Vert _{4/3}\Vert \teta%
_{1}-\teta_{2}\Vert _{4}  \notag \\
& \leq C\left( \sum_{i=2}^{3}\Vert (h_{i})_{1}-(h_{i})_{2}\Vert
_{L^{2}(\Gamma )}^{2}+\Vert f_{1}-f_{2}\Vert _{H}^{2}+\Vert F_{\mathbf{g}%
_{1}}-F_{\mathbf{g}_{2}}\Vert _{4/3}^{2}\right)   \notag \\
& \quad +\frac{1}{2}\Vert \teta_{1}-\teta_{2}\Vert _{V}^{2}.  \notag
\end{align}
Using now \eqref{cdu}, we obtain exactly \eqref{stability}. This concludes
the proof of Theorem~\ref{main}.

\section{Numerical Examples}

In order to obtain numerical results for engineering applications, the
proposed predictive theory has been also implemented in a computational
framework.

The thermomechanical state $(\vec{U}^+, \vec{\beta}^+, T^+)$ after the
collision is obtained by solving Eqs. (\ref{equU}), (\ref{equatemp}) and (%
\ref{thermal}) by means of the finite element method. The post-collision
velocity $\vec{U}^+$ is obtained from Eq. (\ref{equU}). Based on this
solution, the post-collision alloy composition $\vec{\beta}^+$ and
temperature $T^+$ are obtained by solving the coupled Eqs. (\ref{equatemp})
and (\ref{thermal}) through an iterative numerical scheme \cite{mar2013}.

Moreover, the evolution of the thermomechanical state of the structure after
the collision is also predicted. To reach this goal, the constitutive model
for SMA smooth evolution presented in \cite{freroc2009,fremond2011}, and
enriched in \cite{mar2013,mar2014}, has been employed. The model accounts
for the typical shape-memory (i.e., thermal induced transformations) and
pseudoelastic (i.e., stress-induced transformations) effects in SMAs. It is
based on a set of solving equations analogous to Eqs. (\ref{equU}), (\ref%
{equatemp}) and (\ref{thermal}), but formulated under smooth evolution
assumptions. Hence, employed equations allow to compute structure
displacements, as well as the evolution of alloy composition and of the
temperature field, taking the post-collision state $(\vec{U}^+, \vec{\beta}%
^+, T^+)$ as initial condition. The solution strategy for the
post-collision evolution is based on an incremental algorithm which employs
an explicit Euler time discretization (i.e., an updated-Lagrangian
formulation) and a finite-element spatial discretization.

Model parameters are chosen referring to a Ni-Ti alloy: $\rho = 6500$ kg/m$%
^3 $, $l_a = 80$ MJ/(m$^3$), $C = 5.4$ MJ/(m$^3$K), $\lambda = 18$ Ws/(Km),
and $T_o = 332.75$ K \cite{messner2003}. The finite-element discretizations
employ quadratic Lagrange basis functions and a mesh element size equal
about to one cent of the structure maximum size. 

In agreement with the theoretical framework previously described, the collision is assumed to
be adiabatic. Addressing the thermal problem in the smooth evolution after the collision, a convective heat transfer is prescribed on the boundary, with convection coefficient equal to $100$ W/(m$^2$K) and external temperature equal to $T_{ext} = 0.9 T_o$.

\begin{figure}[tbp]
\centering
\includegraphics[width = 0.9\textwidth]{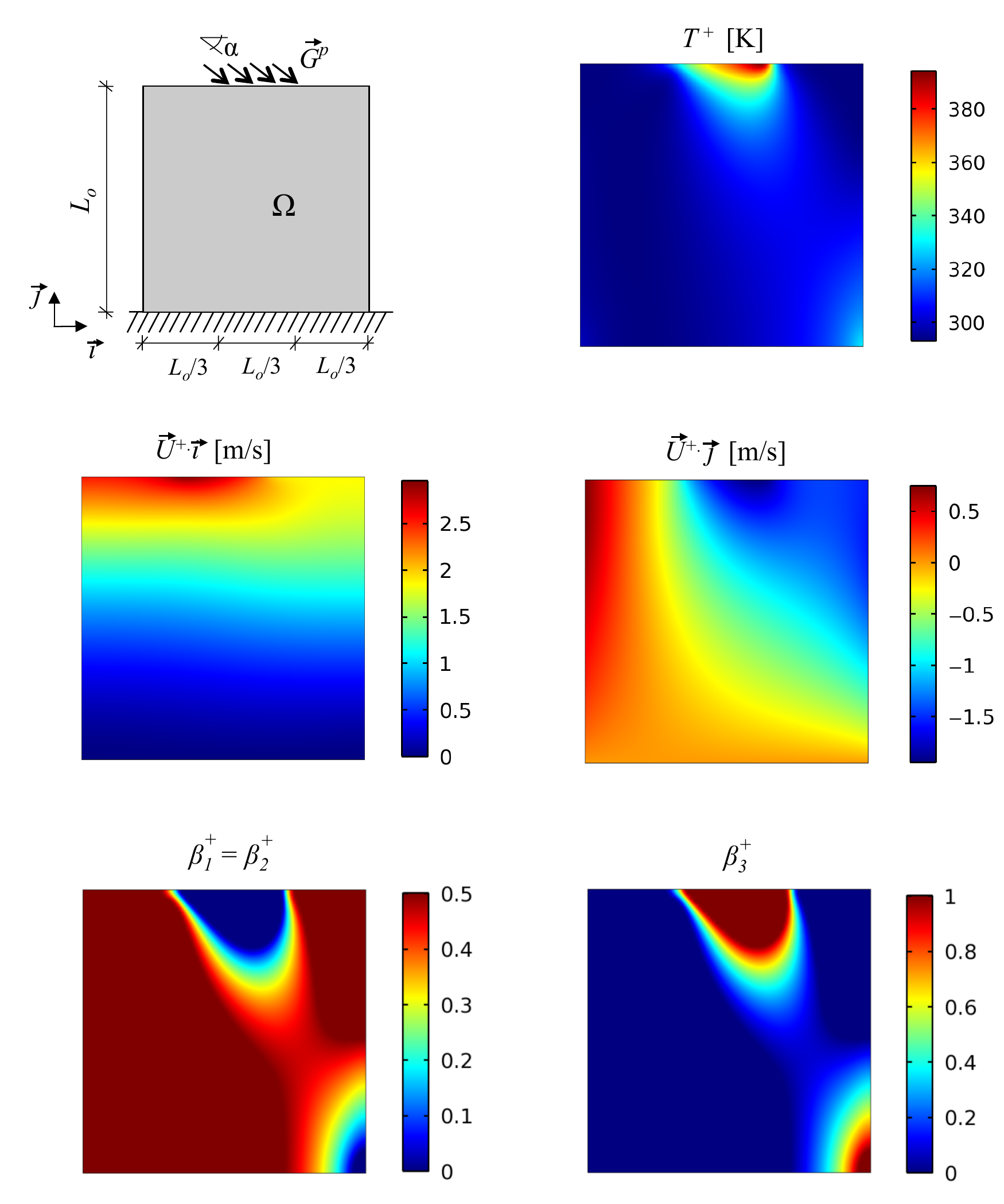}
\caption{A percussion is applied to a 2D solid: post-collision velocity $%
\vec{U}^+$, temperature $T^+$, and martensites/austenite volume fractions $%
\protect\beta_1^+$, $\protect\beta_2^+$, and $\protect\beta_3^+$.
Parameters: $k_v=1$ MPa $\cdot$ s, $c=5\cdot 10^{-2}l_a$, and $k=\protect%
\upsilon=0.5$ MPa, $\|\vec{G}^p\|=20$ MPa $\cdot$ s.}
\label{Figure1}
\end{figure}

\subsection{A Surface Percussion is Applied to a Solid}

We present numerical results for a two-dimensional SMA solid subjected to an
external percussion, \cite{mar2013}. As depicted in Fig. \ref{Figure1}, the
solid is squared in shape ($L_o= 1$ mm wide), fixed on the bottom face to an
immobile obstacle and free on the left and right sides. 

Before the collision, the solid is at rest ($U^-=0$) and at uniform 
temperature $T^- = T_{ext} = 0.9 T_o$, such that the alloy mixture results (for temperatures below the transformation temperature $T_o$, the martensitic phase is at equilibrium): $\beta_1^-=%
\beta_2^-=0.5$ and $\beta_3^-=0$. A percussion $\vec{G}^{p}$ inclined of an
angle $\alpha = 60^{\circ}$ with respect to the horizontal direction is
applied at the center of the top face on a segment with length $L_o/3$. (see
Fig. \ref{Figure1}).

\subsubsection{Velocity and Volume Fractions after the Collision}

As shown in Fig. \ref{Figure1}, a non null velocity field $\vec{U}^+$ is
obtained after the collision. The temperature $T^+$ mainly increases where
the percussion is applied and on the fixed constraint where a percussion
reaction originates (see Fig. \ref{Figure1}). In these regions, the
dissipated work is large, determining the increase of the temperature. In
turn, the latter is coupled with the appearance of the austenite phase $%
\beta_3^+$ (dominant at large temperature) and the disappearance of
martensites $\beta_1^+=\beta_2^+$ (see Fig. \ref{Figure1}).

\begin{figure}[tbp]
\centering
\includegraphics[width = 0.9\textwidth]{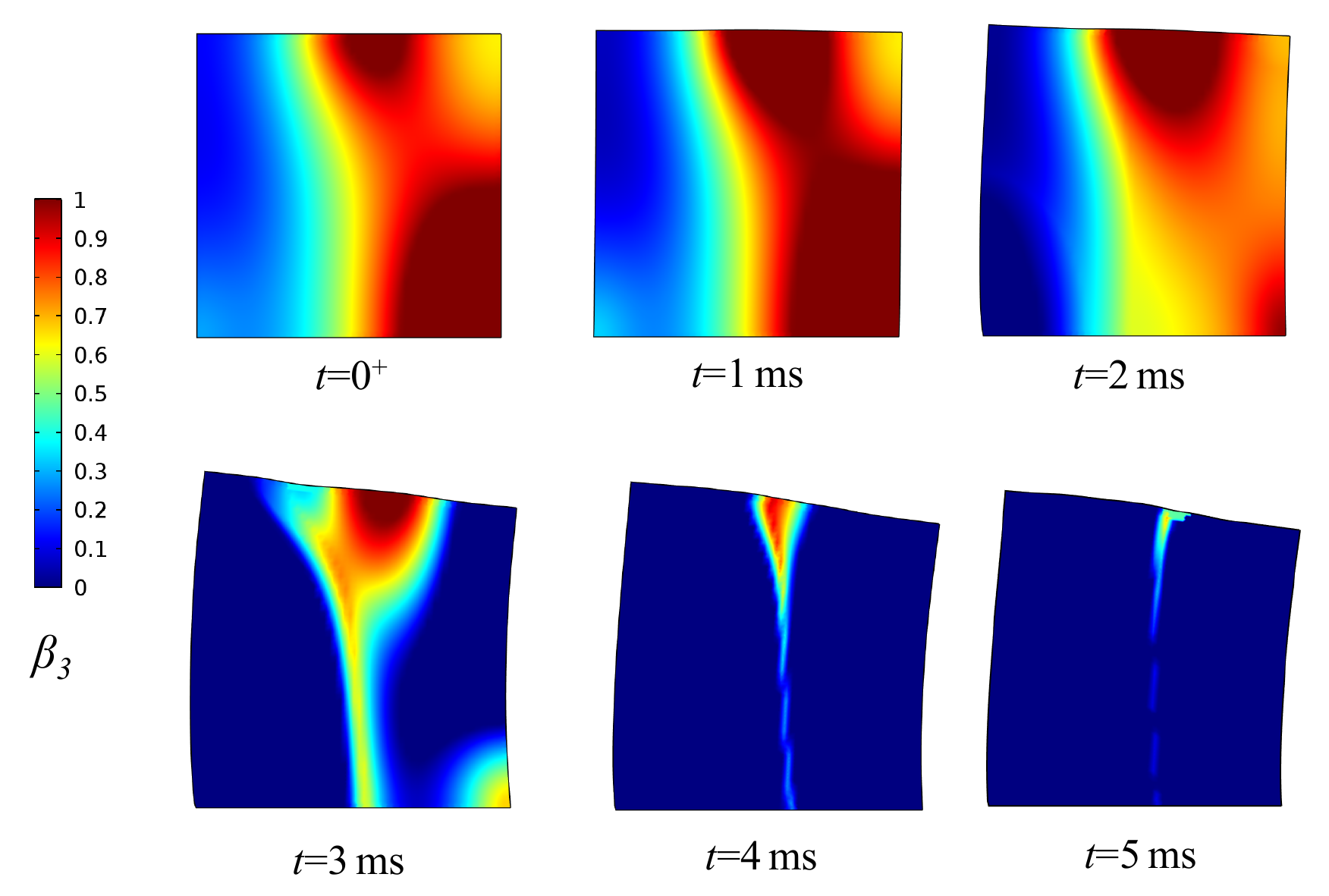}
\caption{A Percussion is Applied to a 2D solid: post-collision smooth
evolution. Austenite volume fraction $\protect\beta_3$ vs. time $t$ in the
current configuration. Parameters: $k_v=1$ MPa $\cdot$ s, $c=5\cdot
10^{-2}l_a$, and $k=\protect\upsilon=0.5$ MPa, $\|\vec{G}^p\|=35$ MPa $\cdot$
s.}
\label{Figure2}
\end{figure}

\subsubsection{Evolution following the Collision: Position and Austenite
Volume Fraction Depending on Time}

The evolution of the thermomechanical state of the structure after the
collision is shown in Fig. \ref{Figure2} where the time-evolution of
the austenite volume fraction $\beta_3$ in the current configuration of the
solid is reported. 

The predicted response is affected by both collision-induced and
stress-induced transformation mechanisms. The former determines indeed a
non-uniform alloy mixture at the beginning of the
post-collision evolution (i.e., $\vec{\beta}^+$). On the other hand,
stress-induced transformations are due to the deformation of the structure
induced by the post-collision velocity field (i.e., $\vec{U}^+$). In
particular, as shown in Fig. \ref{Figure2}, the solid globally rotates and
the collision-induced austenite progressively disappears with
time due to stress-induced phase change. Accordingly, martensites appear,
associated with material pseudoelastic response.

Since material properties after the collision depend on $\vec{\beta}^+$, the
post-collision structural deformation (associated with a non-trivial $\vec{U}%
^+$) is affected by collision-induced transformation mechanisms. Therefore,
collision-induced and stress-induced mechanisms are strongly coupled each
other in determining the post-collision evolution of the
structure, both in terms of deformation and alloy composition. Clearly,
austenite-martensite phase change mechanisms are also coupled with the
evolution of the temperature, whose initial condition after the collision is
given by the non-uniform field $T^+$.

\section*{Acknowledgments}

The financial support of the FP7-IDEAS-ERC-StG \#256872 (EntroPhase) and of
the project Fondazione Cariplo-Regione Lombardia MEGAsTAR ``Matematica
d'Eccellenza in biologia ed ingegneria come acceleratore di una nuova
strateGia per l'ATtRattivit\`a dell'ateneo pavese'' is 
gratefully acknowledged. The paper also benefited from the support of the
GNAMPA (Gruppo Nazionale per l'Analisi Matematica, la Probabilit\`a e le
loro Applicazioni) of INdAM (Istituto Nazionale di Alta Matematica) for ER.
The financial support of the State of Lower Saxony (Germany) in the
framework of the Masterplan ``Smart Biotecs'' is acknowledged by MM.

\end{document}